\documentclass[12pt]{amsart}

\usepackage{amsmath, amssymb, amscd, verbatim, xspace,amsthm}
\usepackage{latexsym, epsfig, color}
\usepackage[hidelinks]{hyperref}

\usepackage[OT2,T1]{fontenc}
\DeclareSymbolFont{cyrletters}{OT2}{wncyr}{m}{n}
\DeclareMathSymbol{\Sha}{\mathalpha}{cyrletters}{"58}

\newcommand{\Z}{\ensuremath{{\mathbb{Z}}}\xspace}

\newcommand{\F}{\ensuremath{{\mathbb{F}}}}

\newcommand{\E}{\ensuremath{{\mathbb{E}}}}

\newcommand{\ra}{\rightarrow}

\newcommand\Hom{\operatorname{Hom}}

\newcommand\Aut{\operatorname{Aut}}
\newcommand\Out{\operatorname{Out}}
\newcommand\im{\operatorname{im}}

\newcommand\Sur{\operatorname{Sur}}

\newcommand\isom{\simeq}
\newcommand\sub{\subset}

\newcommand\GL{\operatorname{GL}}

\newcommand\bq{\begin{equation}}
\newcommand\eq{\end{equation}}

\numberwithin{equation}{section}
\newtheorem{proposition}[equation]{Proposition}
\newtheorem{theorem}[equation]{Theorem}
\newtheorem{corollary}[equation]{Corollary}
\newtheorem{example}[equation]{Example}

\newtheorem{lemma}[equation]{Lemma}

\theoremstyle{remark}
\newtheorem{remark}[equation]{Remark}

\usepackage{fullpage,pdfsync}

\newenvironment{definition}{\vspace{2 ex}{\noindent{\bf Definition. }}}{\vspace{2 ex}}

\newtheorem{nts}{Note to self}
\newcommand{\melanie}[1]{{\color{blue} \sf $\clubsuit\clubsuit\clubsuit$ Melanie: [#1]}}
\newcommand{\yuan}[1]{{\color{green!60!black} \sf $\clubsuit\clubsuit\clubsuit$ Yuan: [#1]}}

\newcommand\hF{\hat{F}}
\newcommand{\Prob}{\operatorname{Prob}}

\usepackage{faktor}
\usepackage{enumitem}
\usepackage{tikz-cd}
\usepackage{longtable}
\newcommand\Cen{\mathbf{C}}
\newcommand\bS{\bar{S}}
\newcommand\bT{\bar{T}}
\newcommand\Inn{\operatorname{Inn}}
\newcommand\CF{\mathcal{CF}}
\allowdisplaybreaks

\title{The free group on $n$ generators modulo $n+u$ random relations as $n$ goes to infinity}

\author{Yuan Liu}
\address{Department of Mathematics\\
University of Wisconsin-Madison \\ 480 Lincoln Drive \\
Madison, WI 53705 USA}  
\email{yliu@math.wisc.edu}

\author{Melanie Matchett Wood}
\address{Department of Mathematics\\
University of Wisconsin-Madison \\ 480 Lincoln Drive \\
Madison, WI 53705 USA\\
and
American Institute of Mathematics\\600 East Brokaw Road\\
San Jose, CA 95112 USA}  
\email{mmwood@math.wisc.edu}

\begin{document}

\begin{abstract}
We show that, as $n$ goes to infinity, the free group on $n$ generators, modulo $n+u$ random relations,  converges to a random group that we give explicitly.  This random group is a non-abelian version of the random abelian groups that feature in the Cohen-Lenstra heuristics.  For each $n$, these random groups belong to the few relator model in the Gromov model of random groups.
\end{abstract}

%Key words: Cohen-Lenstra heuristics, Gromov model, random groups, few relator model, balanced groups

\maketitle
\vspace{-36pt}
\section{Introduction}\label{S:intro}
For an integer $u$ and positive integers $n$, we study the random group given by the free group $F_n$ on $n$ generators modulo $n+u$ random relations.  In particular we find that these random groups have a nice limiting behavior as $n\ra\infty$ and we explicitly describe the limiting random group.

There are two ways to take relations in a ``uniform'' way: 1)complete $F_n$ to the profinite free group $\hF_n$ on $n$ generators and take relations with respect for Haar measure, or 2)take relations from $F_n$ uniformly among words up to length $\ell$ and then let $\ell\ra\infty$.  In Proposition~\ref{P:same}, we show that the random groups obtained from the second method weakly converge, as $\ell\ra\infty$, to the random groups obtained from the first method. 

For a positive integer $n$, let $\hF_n$ be the profinite free group on $n$ generators.  For an integer $u$,  we define the random group $X_{u,n}$ by taking the quotient of $\hF_n$ by (the closed, normal subgroup generated by) $n+u$ independent random generators, taken from Haar measure on $\hF_n$.  We need to define a topology to make precise the convergence of  $X_{u,n}$ as $n\ra\infty$.

Let $S$ be a set of (isomorphism classes of) finite groups.  Let $\bar{S}$ be the smallest set of groups containing $S$ that is closed under taking quotients, subgroups, and finite direct products.  For a profinite group $G$, we write $G^{\bar{S}}$ for its pro-$\bar{S}$ completion.  
\begin{comment}
In the free group in countably many generators, we let $W_S$ be the elements that are in the kernel of every map to an element of $S$; i.e. $W_S$ is the set of words in formal variables and their inverses that vanish for every specialization of the variables to every group in $S$.  For a profinite group $G$, we write $Q_S(G)$ for the quotient of $G$ by (the closed, normal subgroup generated by) all specializations of words in $W_S$ to $G$ \melanie{in fact, this should already be a normal subgroup}.  For a positive integer $N$, we write $Q_N(G)$ for $Q_{S_N}(G)$, where $S_N$ is the set of all groups of order at most $N$.
\end{comment}
We consider the set $\mathcal{P}$ of isomorphism classes of profinite groups $G$ such that $G^{\bar{S}}$ is finite for all finite sets $S$ of finite groups.  All finitely generated profinite groups are in $\mathcal{P}$ %(see Lemma~\ref{L:sfinite}) 
and all groups in 
$\mathcal{P}$ are small in the sense of \cite[Section 16.10]{Fried2008}.
We define a topology on $\mathcal{P}$ in which the basic opens are, for each finite set $S$ of finite groups and finite group $H$, the sets $U_{S,H}:=\{G \mid G^{\bar{S}} \isom H\}$.

\begin{theorem}\label{T:Main}
Let $u$ be an integer.  Then there is a probability measure $\mu_u$ on $\mathcal{P}$ for the $\sigma$-algebra of Borel sets such that
as $n\ra\infty$, the distributions of $X_{u,n}$ weakly converge to $\mu_u$.
\end{theorem}

We give these $\mu_u$ explicitly in fact.  See Equation~\eqref{E:limval} for a formula for $\mu_u$ on each basic open, and see Section~\ref{S:examples} for several other interesting examples of the values of these measures.  In fact, we prove in Theorem~\ref{T:calcinf} a stronger form of convergence than weak convergence, which, in particular, tell us the measure of any  finite group.  In particular, we have
\begin{equation}\label{E:trivialgp}
\mu_u(\textrm{trivial group})=\prod_{\substack{G \textrm{ finite simple}\\\textrm{abelian group}}}
\prod_{i=u+1}^{\infty} (1-|G|^{-i})
 \prod_{\substack{G \textrm{ finite simple}\\\textrm{non-abelian group}}} e^{-|\Aut(G)|^{-1}|G|^{-u}},
\end{equation}
which is  $\approx.4357$ when $u=1$.
The abelian group version of this problem has been well-studied, as the limiting groups when $u=0,1$ are the random groups of the Cohen-Lenstra heuristics.  The first factor above, as a product over primes, is very familiar from the random groups of the Cohen-Lenstra heuristics, but here it naturally appears as part of a product over all finite simple groups.

Cohen and Lenstra \cite{Cohen1984} defined certain random abelian groups that they predicted gave the distribution of class groups of random quadratic fields.  Friedman and Washington \cite{Friedman1989} later realized that these random abelian groups arose as the limits of cokernels of random matrices, which is just a rewording of the abelianization of our construction above.  These random abelian groups are universal, in the sense that, as $n\ra\infty$, taking $\Z^n$ modulo almost any collection of $n+u$ independent relations will give these same random abelian groups, even if the relations are taken from strange and lopsided distributions \cite{Wood2015a}.

One motivation for our work is to develop a non-abelian version of the random abelian groups of Cohen and Lenstra, in order to eventually be able to model non-abelian versions of class groups of random number fields.  Boston, Bush, and Hajir \cite{Boston2016a} have defined random pro-$p$ groups that they conjecture model the pro-$p$ generalizations of class groups of random imaginary quadratic fields.  In their definition, they were able to use special properties of $p$-groups to give a definition that avoids the limit as $n\ra\infty$ that we study above (or rather, reduces the question of the limit as $n\ra\infty$ to the abelian case, which was already understood).

There is a large body of work on the Gromov, or density, model of random groups (see \cite{Ollivier2005} for an excellent introduction).  In this model, one takes $F_n$ modulo $r(\ell)$ random relations uniform among words of length $\ell$, and studies the behavior as $\ell\ra\infty$. 
When $r(\ell)$ grows like $(2n-1)^{d\ell}$ this is called the density $d$ model.  
There has been a great amount of work to understand, as $\ell\ra\infty$, what properties hold asymptotically almost surely for these groups (e.g see \cite{Ollivier2005, Ollivier2010} for an overview and \cite{Calegari2015,Kotowski2013,Mackay2016,Ollivier2011} for some more recent examples).  Our $X_{u,n}$ are limits as $\ell\ra\infty$ of density $0$ models of these random groups. 
In fact, slightly different density $0$ models introduced by Gromov \cite{Gromov1987} and Arzhantseva and Ol'shanskii \cite{Arzhantseva1996} were predecessors to the work on the density model for arbitrary $d$ and are also the subject of a large body of work.
 However, the emphasis of our work is different from much of the previous work on these models of random groups.  That work has often emphasized a random group with given generators, and we consider only the isomorphism class of the group and focus on the convergence to a limiting random variable as the number of generators goes to infinity.  Properties that hold asymptotically almost surely as $\ell\ra\infty$ may not hold of the limiting random variable.  For example, in our topology, any Gromov random group with $r(\ell)\ra\infty$ weakly converges to the trivial group, yet at low enough density these groups are asymptotically almost surely not trivial (see Proposition~\ref{P:alltrivial} and Remark~\ref{R:alltrivial}).  Our topology is  aimed at understanding finite quotients of groups, and is rather different than the topology due to Chabauty \cite{Chabauty1950} and Grigorchuk \cite{Grigorchuk1984} on the space of marked groups that emphasizes the geometry of the Cayley graphs but isn't well behaved on isomorphism classes of groups.

The previous work with the closest emphasis to ours is that of Dunfield and Thurston \cite{Dunfield2006}.  They studied $F_n$ modulo $r$ random relations (with both methods described above of taking relations)  in order to contrast those random groups with random $3$-manifold groups.  Their main consideration was the probability that these random groups (for fixed $n$ and $r$) had a quotient map to a fixed finite group.  They do observe \cite[Theorem 3.10]{Dunfield2006} that for a fixed non-abelian finite simple group $G$, the distribution of the number of quotient maps to $G$ has a Poisson limiting behavior as $n\ra \infty$; this is the first glimpse of the nice limiting behavior as $n\ra\infty$ that we study in this paper.

Jarden and Lubotzky \cite{Jarden2006} studied the normal subgroup of $\hF_n$ generated by a fixed number of random elements, in particular proving that when it is infinite index that it is almost always the free profinite group on countably many generators.  Our work here complements theirs, as they have determined the structure of the random normal subgroup and we determine the structure of the quotient by this random normal subgroup.

The bulk of this paper is devoted to showing the existence of the measure $\mu_u$ of Theorem~\ref{T:Main}.  Let $\mu_{u,n}$ be the distribution of our random group $X_{u,n}$.  Since the basic opens in our topology are also closed, it is clear that if $\mu_u$ exists, then for any basic open $U$ we have $\mu_u(U)=\lim_{n\ra\infty} \mu_{u,n} (U).$
The argument for the existence of $\mu_u$ breaks into two major parts.  The first part is to show the limit
$\lim_{n\ra\infty} \mu_{u,n} (U)$ exists.  The second part is to show that these measures on basic opens define a countably additive measure.  After giving some notation and basic definitions in Section~\ref{S:Notations}, we will give the values of $\mu_u$ on basic opens in Section~\ref{S:definemu} for easy reference.   Then in Section~\ref{S:setup} we set up the strategy for proving 
$\lim_{n\ra\infty} \mu_{u,n} (U)$ exists, which is entirely group theoretical.  This argument will take us through Section~\ref{S:basics}.  It is easy to express 
$\mu_{u,n}$ in group theory terms involving $\hat{F}_n$.  However, such expressions do not allow one to take a limit as $n\ra\infty$, and so the main challenge is to extract $\hat{F}_n$ from the description of the probabilities so that they only involve the number $n$ and  group theoretical quantities that do not depend on $n$.  This requires several steps.  In Section~\ref{S:genprob}, we express the probabilities in terms of multiplicities of certain groups appearing in $\hat{F}_n$.  In Section~\ref{S:factors}, we bound what possible groups can have positive multiplicities.  In Section~\ref{S:countqt}, we relate the multiplicities to a count of certain surjections, and  finally in Section~\ref{S:basics} we count these surjections in another way that eliminates $\hat{F}_n$ from our description of the probabilities.

The next challenge is to show the countable additivity of the $\mu_u$ that we have then defined on basic opens. It follows from Fatou's lemma that for a finite set $S$ of finite groups,
$$
\sum_{H \text{ is finite}} \lim_{n\ra\infty} \mu_{u,n} (U_{S, H}) \leq 1. 
$$
However, a priori, this inequality may be strict.  In the limit as $n$ goes to infinity there could be escape of mass.
To show that this does not occur, we require bounds on the $\mu_{u,n} (U_{S, H})$ that are sufficiently uniform in $n$.  The difficultly is that our group theoretical expressions do not easily lend themselves to the kind of bounds useful for an analytic argument.  We obtain the necessary uniformity by considering a notion of \emph{chief factor pairs}, which generalizes the notion of a chief factor of a group to also include the conjugation action on the chief factor.  We are able to bound the size of the outer action of conjugation on chief factors for a given $S$ in Section~\ref{S:factors}, which then, combined with an induction on $S$, gives us the uniformity necessary to show in Section~\ref{S:cntadd} that the above inequality is actually an equality.  That is the heart of the proof of countable additivity, which we show in Section~\ref{S:cntadd}.

Once we have established the existence of the measure $\mu_u$ with the desired measure on basic opens, Theorem~\ref{T:Main} follows immediately in Section~\ref{S:mainproof}.  In Section~\ref{S:arbitraryS}, we give the measures of sets of the form $\{X \in\mathcal{P} \mid X^{\bS}\simeq H^{\bS}\}$ for arbitrary sets $S$ of finite groups, and see that $\mu_u$ and $\lim_{n\ra\infty} \mu_{u,n}$ agree there, giving a stronger convergence than the weak convergence of Theorem~\ref{T:Main}.  There, one inequality is automatic, and we then we argue that we either have the necessary uniformity to get equality, or that the larger probability is $0$, which also gives equality.
The result of Section~\ref{S:arbitraryS} then allows us to compute measures of many different Borel sets, and in Section~\ref{S:examples}, we give many examples including the trivial group, infinite groups, and distributions of the abelianization and pro-nilpotent quotient.  In Section~\ref{S:prob0}, we see that the measures $\mu_u$ give positive measure to any basic open where groups can be generated by $u$ more relations than generators. Finally, in Section~\ref{S:non-profinite}, we compare the profinite model used in this paper to the discrete group model described above. 

This is the beginning of investigation into these random groups, and there are many further questions we would like to understand.  Are these measures universal in the sense of \cite{Wood2015a}, i.e. would we still get $\mu_u$ as $n\ra\infty$ even if we took our relations from a different measure?  Are these measures determined by their moments, which in \cite[Lemma 18]{Heath-Brown1994-1}, \cite[Section 4.2]{Fouvry2006}, \cite[Lemma 8.2]{Ellenberg2016}, \cite[Theorem 8.3]{Wood2017}, and \cite[Theorem 1.4]{Boston2017} has been an important tool to identify analogous random groups?  What is the measure of the set of all infinite groups when $u\geq 0$ (see Example~\ref{Ex:inf}, and note by \cite{Jarden2006} this implies the normal subgroup generated by the relations is free on countably many generators with probability $1$)?  What is the measure of the set of finitely generated groups, and of finitely presented groups?  Do the $\mu_{u,n}$ converge strongly to $\mu_u$?  Besides their inherent interest, many of these questions have implications for the possible connections to number theory described above.

\section{Notation and basic group theoretical definitions}\label{S:Notations}

\subsection{Notation}\label{SS:2.1}
In Section~\ref{S:index}, we give a list of symbols used in the paper in more than one section for ease of reference.

Whenever we take a quotient by relations, we always mean by the closed, normal subgroup generated by those relations.  For elements $x_1,\dots$ of a group $G$, we write $[x_1,\dots]_G$ for the closed normal subgroup of $G$ generated by $x_1,\dots$.

We write $G\isom H$ to mean that $G$ and $H$ are isomorphic.  For profinite groups, we always mean isomorphic as profinite groups.  For two groups $G$ and $H$, we write $G=H$ when there is an obvious map from one of $G$ or $H$ to the other (e.g. when $H$ is defined as a quotient or subgroup of $G$) and that map is an isomorphism.

For a group $G$, we write $G^j$ for the direct product of $j$ copies of $G$. If $H$ is a subgroup of $G$, then we denote the centralizer of $H$ by $\Cen_G(H)$.

When we say a set of finite groups, we always mean a set of isomorphism classes of finite groups.
\subsection{$F$-groups}\label{SS:2.2}

If $F$ is a group, an \emph{$F$-group} is a group $G$ with an action of $F$. 
%not sure if I should use this condition:
%such that the image of $F\ra\Aut(G)$ includes $\operatorname{Inn}(G)$.  
A morphism of $F$-groups is a group homomorphism that respects the $F$-action.    An $F$-subgroup is a subgroup $G$ such that $f(G)=G$ for all $f\in F$, and an $F$-quotient is a group quotient homomorphism that respects the $F$-action. 
%If use extra condition:
%Note that every $F$-subgroup is automatically normal.
 An \emph{irreducible $F$-group} is an $F$-group with no normal 
 $F$-subgroups except the trivial subgroup and the group itself. 
  We write $\Hom_F(G_1,G_2)$ for the $F$-group morphisms from $G_1$ to $G_2$ and $h_F(G):=|\Hom_F(G,G)|$. We write $\Sur_F(G_1, G_2)$ for the $F$-group surjections from $G_1$ to $G_2$, and $\Aut_F(G)$ for the $F$-group automorphisms of $G$.
 For a sequence $x_k$ in an $F$-group $G$, let $[x_1,\dots ]_F$ be the closed normal  $F$-subgroup of $G$  generated by the $x_k$.

\subsection{$H$-extensions} \label{SS:2.3}

For a group $H$, an \emph{$H$-extension} is a group $E$ with a surjective morphism $\pi: E\ra H.$  
If $\pi:E\ra H$ and $\pi':E'\ra H$ are $H$-extensions, a morphism from $(E,\pi)$ to $(E',\pi')$ is a group homomorphism $f:E\ra E'$ such that $\pi=\pi' \circ f$.  If $\pi:E\ra H$ and $\pi':E'\ra H$ are $H$-extensions, we write $\Sur_H(\pi,\pi')$ for the set of surjective morphisms from $(G,\pi)$ to $(G',\pi')$.
For an $H$-extension $E$, we write $\Aut_H(E,\pi)$ for the automorphisms of $(E,\pi)$ as an $H$-extension.  
If $(E,\pi)$ is an $H$-extension, a sub-$H$-extension is a subgroup $E'$ of $E$ with $\pi|_{E'}$, such that $\pi|_{E'}$ is surjective.
Note that when $\ker \pi$ is abelian, it is an $H$-group under conjugation in $E$.

\subsection{Pro-$\bS$ completions and level $S$ groups}\label{SS:proSdef}

Given a set $S$ of finite groups, we let $\bar{S}$ denote the smallest set of groups containing $S$ that is closed under taking quotients, subgroups and finite direct products. (This is called the variety of groups generated by $S$.) Given a profinite group $G$, we write $G^{\bS}$ for its pro-$\bS$ completion, which is defined as 
$$G^{\bS} = \varprojlim_{M} G/M,$$
where the inverse limit is taken over all closed normal subgroups $M$ of $G$ such that $G/M\in \bS$.

\begin{definition}
 For a set $S$ of finite groups, we say that a profinite group $G$ is \emph{level $S$} if $G \in \bS$. Also, for a positive integer $\ell$, let $S_\ell$ be the set consisting of all groups whose order is less than or equal to $\ell$. Then we say $G$ is \emph{level $\ell$} if $G\in \bS_\ell$.  Note that for $G\in\mathcal{P}$ we have that $G$ is level $S$ if and only if $G=G^{\bS}$.
\end{definition}

\section{Definition of $\mu_u$}\label{S:definemu}

For integers $n\geq 1$ and $u>-n$, recall that $X_{u,n}$ is the random group defined by taking quotient of the free profinite group $\hat{F}_n$ on $n$ generators by $n+u$ independent random relations that are taken from the Haar measure on $\hat{F}_n$. For finite set  $S$ of finite groups and finite group $H$, let $U_{S,H}:=\{X\in \mathcal{P} \,|\, X^{\bS}\isom H \}$
(where $\mathcal{P}$ is the set of isomorphism classes of profinite groups $G$ such that the pro-$\bar{S}$ completion $G^{\bar{S}}$ is finite for all finite sets $S$ of finite groups).
 We have a measure $\mu_{u,n}$ on the $\sigma$-algebra of Borel sets of $\mathcal{P}$ such that $\mu_{u,n}(A)=\Prob(X_{u,n}\in A)$.
We will define a measure $\mu_u$, for each integer $u$, at first as a measure on the algebra $\mathcal{A}$ of sets generated by the $U_{S,H}$.  For $A\in \mathcal{A}$, we define
\begin{equation}\label{E:defmu}
\mu_u(A):=\lim_{n\ra\infty} \mu_{u,n}(A).
\end{equation}
We will below establish that 1) this limit exists when $A=U_{S,H}$ (see Theorem~\ref{T:calc}, and Equation~\eqref{E:limval} just below, in which we give the value of the limit), and hence for any $A\in \mathcal{A}$ since the limit is compatible with finite sums and subtraction from $1$; and 2) $\mu_u$ is countably additive on $\mathcal{A}$ (see Theorem~\ref{T:countadd}).  These two results represent the bulk of the work of the paper. Then by Carath\'{e}odory's extension theorem, it follows that $\mu_u$ extends uniquely to a probability measure on $\mathcal{P}$.

\subsection{Value of $\mu_u$ on basic open sets}\label{SS:basicq}

Given a finite group $H$, let $\mathcal{A}_H$ be the set of isomorphism classes of non-trivial finite abelian irreducible $H$-groups.    Let $\mathcal{N}$ be the set of isomorphism classes of finite groups that are isomorphic to $G^j$ for some finite simple non-abelian group $G$ and a positive integer $j$.
Let $S$ be a set of finite groups, and $H$ a finite level $S$ group.  
For $G\in \mathcal{A}_H$, we define the quantity
$$
\lambda(S,H,G):=(h_H(G)-1) 
\sum_{\substack{\textrm{isom. classes of $H$-extensions $(E,\pi)$}\\ \textrm{such that $\ker \pi \simeq G$ as  $H$-groups,} \\
\textrm{and $E$ is level $S$}  }}
\frac{1}{|\Aut_H(E,\pi)|}.
$$
We will see in Remark~\ref{R:lamint} that for $G\in \mathcal{A}_H$, the number $\lambda(S,H,G)$ is an integer power of $h_H(G)$.
If $G\in \mathcal{N}$, we define
$$
\lambda(S,H,G):=\sum_{\substack{\textrm{isom. classes of $H$-extensions $(E,\pi)$}\\ \textrm{such that $\ker \pi \simeq G^j$ as groups, }\\
\textrm{$\ker \pi$ irred. $E$-group,} \\
\textrm{and $E$ is level $S$} }} 
\frac{1}{|\Aut_H(E,\pi)|}.
$$
The definitions are not quite parallel in the abelian and non-abelian cases, but this is unavoidable given the different behavior of abelian and non-abelian simple groups.

It will follow from Theorem~\ref{T:calc} below that for a finite set $S$ of finite groups and a finite level $S$ group $H$, we have
\begin{align}\label{E:limval}
\mu_u(U_{S,H})=&\frac{1}{|\Aut(H)||H|^{u}}\prod_{\substack{G \in \mathcal{A}_H}} 
\prod_{i=0}^{\infty} (1-\lambda(S,H,G)
\frac{h_H(G)^{-i-1}}{ |G|^{{u}}})
 & \prod_{\substack{ G\in \mathcal{N}}}  e^{-|G|^{-u}\lambda(S,H,G)}.
\end{align}
Theorem~\ref{T:calcinf} gives the analogous result for an infinite set $S$.
We will see in Section~\ref{S:factors} that for finite $S$ only finitely many elements of $\mathcal{A}_H$ and $\mathcal{N}$ contribute non-trivially to this product.

\section{Setup and organization of the proofs}\label{S:setup}

The proof of Equation~\eqref{E:limval} will be established from Section~\ref{S:genprob} to Section~\ref{S:basics}, which are dominated by group theoretical methods. Here we outline the proof  for the reader's convenience.

Suppose $n$ is a positive integer, $S$ is a finite set of finite groups, and $H$ is a finite level $S$ group. Then $(\hat{F}_n)^{\bS}$ is a finite group \cite[Cor. 15.72]{Neumann1967} and $(X_{u,n})^{\bS}$ has the same distribution as the quotient of $(\hat{F}_n)^{\bS}$ by $n+u$ independent, uniform random relations $r_1, \cdots r_{n+u}$ from $(\hat{F}_n)^{\bS}$. By the definition of $\mu_u$, we have that 
$$\mu_u(U_{S,H})= \lim_{n\to \infty} \Prob((\hat{F}_n)^{\bS}/[r_1,\cdots, r_{n+u}]_{(\hat{F}_n)^{\bS}}\simeq H).$$

%\melanie{I don't think these should be here.  We can move them to where they were later, and just refer forward to the lemma.}

We consider a normal subgroup $N$ of $(\hat{F}_n)^{\bS}$ with an isomorphism $(\hat{F}_n)^{\bS}/N \simeq H$.  Let $M$ be the intersection of all maximal proper $(\hat{F}_n)^{\bS}$-normal subgroups of $N$. We denote $F=(\hat{F}_n)^{\bS}/M$ and $R=N/M$. Then for independent, uniform random elements $r_1, \cdots, r_{n+u}$ of $(\hat{F}_n)^{\bS}$, we have that $[r_1, \cdots, r_{n+u}]_{(\hat{F}_n)^{\bS}}=N$ if and only if $R$ is the normal subgroup generated by the images of $r_1, \cdots, r_{n+u}$ in $F$. Indeed, the ``only if'' direction is clear; and if $[r_1, \cdots, r_{n+u}]_{(\hat{F}_n)^{\bS}}/M = R$, then $[r_1, \cdots, r_{n+u}]_{(\hat{F}_n)^{\bS}}=N$ since $[r_1, \cdots, r_{n+u}]_{(\hat{F}_n)^{\bS}}$ being contained in a proper maximal $(\hat{F}_n)^{\bS}$-normal subgroup of $N$ would imply that its image is contained in a proper maximal $F$-normal subgroup of $R$. 

Any two surjections from $(\hat{F}_n)^{\bS}$ to $H$ are isomorphic as $H$-extensions \cite[Proposition 2.2]{Lubotzky2001}. Thus, the short exact sequence 
\begin{equation}\label{E:fundSES}
1 \ra R \ra F \ra H \ra 1
\end{equation}
does not depend (up to isomorphisms of $F$ as an $H$-extension) on the choice of the normal subgroup $N$. 

\begin{definition}
  Given a finite set $S$ of finite groups, a positive integer $n$ and a finite level $S$ group $H$, the short exact sequence defined in Equation~\eqref{E:fundSES} is called \emph{the fundamental short exact sequence associated to $S$, $n$ and $H$}.
\end{definition}

By the above arguments, $\Prob((\hat{F}_n)^{\bS}/[r_1, \cdots, r_{n+u}]_{(\hat{F}_n)^{\bS}}\simeq H)$ equals  the number of normal subgroups $N$ of $(\hat{F}_n)^{\bS}$ with $(\hat{F}_n)^{\bS}/N\simeq H$ times the probability that independent, uniform random elements $x_1, \cdots, x_{n+u}\in F$ normally generate $R$. Note that the number of such normal subgroups $N$ is $|\Sur((\hat{F}_n)^{\bS}, H)|/|\Aut(H)|$, and there is a one-to-one correspondence between $\Sur((\hat{F}_n)^{\bS}, H)$ and $\Sur(\hat{F}_n, H)$. It follows that
\begin{equation}\label{E:setup}
\Prob((\hat{F}_n)^{\bS}/[r_1, \cdots, r_{n+u}]_{(\hat{F}_n)^{\bS}}\simeq H) = \frac{|\Sur(\hat{F}_n, H)|}{|\Aut(H)|} \Prob([x_1, \cdots, x_{n+u}]_F = R).
\end{equation}

It therefore suffices to compute $\Prob([x_1, \cdots, x_{n+u}]_F=R)$. 
Note that $R$ is an $F$-group under the conjugation action. It will follow from Lemma~\ref{L:irred} that $R$ is a direct product of irreducible $F$-groups. 
Theorem~\ref{T:probfrommult} will prove the formula for $\Prob([x_1, \cdots, x_{n+u}]_F=R)$  for $F$ and $R$ where $R$ is a direct product of irreducible $F$-groups, in terms of the multiplicities of the various irreducible $F$-group factors of $R$.  In Section~\ref{S:factors}, we will give some criteria for which irreducible $F$-groups can appear in $R$. Then in Section~\ref{S:countqt}, we will relate the multiplicities of irreducible factors in $R$ to the number of normal subgroups of $R$ with specified quotients. In Section~\ref{S:basics}, we will count these normal subgroups of $R$ in another way in order to finally give an explicit formula for $\mu_{u,n}(U_{S,H})$.  This formula will be explicit enough that we can easily take the limit as $n\ra\infty$, giving
Equation~\eqref{E:limval}.

\section{Generating probabilities for products of irreducible $F$-groups} \label{S:genprob}

Throughout this section, we let $n\geq 1$ and $u>-n$ be integers, $F$ a group, and $R$ a finite product of finite irreducible $F$-groups. (We don't require $R$ to be a subgroup of $F$.) The goal of this section is to prove the following theorem which gives the probability that the normal $F$-subgroup generated by $n+u$ random elements of $R$ is the whole group.
 
\begin{theorem}\label{T:probfrommult}
Let $F$ be a group and $G_i$ be finite irreducible $F$-groups for $i=1,\dots, k$ such that for $i\ne j$, we have that $G_i$ and $G_j$ are not isomorphic $F$-groups, and let $m_i$ be non-negative integers. 
   Let $R=\prod_{i=1}^k G_i^{m_i}$.
Then
$$
\Prob( [x_1,\dots, x_{n+u} ]_F=R)=\prod_{\substack{1\leq i\leq k\\ G_i \textrm{ abelian}}} 
\prod_{j=0}^{m_i-1} (1-h_F(G_i)^j |G_i|^{-n-u})
 \prod_{\substack{1\leq i\leq k\\ G_i \textrm{ non-abelian}}}  (1-|G_i|^{-n-u})^{m_i}
$$ 
where the $x_i$ are independent, uniform random elements of $R$.
\end{theorem}

\begin{remark}\label{R:Gpow}
Given a finite abelian irreducible $F$-group $G$, if we let $\mathfrak{m}$ be maximal such that $G^\mathfrak{m}$ can be generated by one element as an $F$-group, then we have
$h_F(G)^\mathfrak{m}=|G|$.  This follows from Theorem~\ref{T:probfrommult} because if we take $m_i=\mathfrak{m}$, the probability that one element generates $G^{m_i}$ is positive, but if we take $m_i=\mathfrak{m}+1$ the probability is $0$. 
\end{remark}

We will build up to Theorem~\ref{T:probfrommult} through several lemmas.
First, we determine the structure of normal $F$-subgroups of products of  irreducible $F$-groups. 

\begin{lemma}\label{L:subprod}
If $G_i$ are irreducible $F$-groups and $N$ is an $F$-subgroup of $\prod_{i=1}^m G_i$ that projects to $1$ or $G_i$ in each factor, then there exists a subset $J\sub\{1,\dots,m\}$ such that the projection of $N$ to $\prod_{i\in J} G_i$ is an isomorphism.
\end{lemma}
\begin{proof}
We prove this by induction on $m$.  Let $\pi_m$ be the projection map from $\prod_{i=1}^m G_i$ to $G_m$, and $\pi$ the projection map from $N$ to $\prod_{i=1}^{m-1} G_i$.  Since $\pi_m(N)$ is $1$ or $G_m$, and $\pi_m(\ker \pi)$ is a normal $F$-subgroup of
$\pi_m(N)$, we have $\pi_m(\ker \pi)$ is $1$ or $G_m$.  If $\pi_m(\ker \pi)=1$, then since $\ker \pi\cap \ker \pi_m=1$, we have $\ker \pi=1$ and $N$ is isomorphic to $\pi(N)$. 
If $\pi_m(\ker \pi)=G_m$, then $N$ is isomorphic to $\pi(N) \times G_m$.
In either case, we apply the inductive hypothesis to $\pi(N)$ and conclude the lemma.  
\end{proof}

\begin{lemma}\label{L:preSchur}
Let $G_1$ and $G_2$ be irreducible $F$-groups.  Then any homomorphism of $F$-groups $\phi: G_1\ra G_2$ with normal image is either trivial or an isomorphism.
\end{lemma}
\begin{proof}
If it is not trivial, then $\ker(\phi)$ is a normal $F$-subgroup and so
must be trivial, and $\im(\phi)$ is a normal $F$-subgroup and
must be $G_2$, so it is a bijection.
\end{proof}

\begin{lemma}\label{L:subsplit}
Let $G_i$ be irreducible $F$-groups for $i=1,\dots, k$ such that for $i\ne j$, we have that $G_i$ and $G_j$ are not isomorphic as $F$-groups. 
Let $N$ be a normal $F$-subgroup of  $\prod_{i=1}^k G_i^{m_i}$, then $N=\prod_{i=1}^k N_i,$
where $N_i$ is a normal $F$-subgroup of $G_i^{m_i}$.
\end{lemma}
\begin{proof}
Since $N$ is a normal $F$-subgroup of $\prod_{i=1}^k G_i^{m_i}$, its projection to each factor $G_i$ is normal $F$-subgroup of $G_i$, hence it's either 1 or $G_i$. By Lemma~\ref{L:subprod}, we can write $N$ abstractly as $\prod_{i=1}^k G_i^{n_i}$ and  define $N_i$ to be the subgroup of $N$ such that it is the image of the factor $G_i^{n_i}$ under the chosen isomorphism between $\prod_{i=1}^k G_i^{n_i}$ and $N$.
From Lemma~\ref{L:preSchur}, we see that for $i\neq j$ the projection $N_i\to G_j^{m_j}$ is trivial, and it follows that $N_i$ is the subgroup of elements of $N$ that are trivial in the projections to $G_j^{m_j}$ for all $j\ne i$.  Finally, if $n\in N_i$, then we can see that any $\prod_{i=1}^k G_i^{m_i}$ conjugate of $n$ is trivial in the projections to $G_j^{m_j}$ for all $j\ne i$ and in is $N$.  Hence $N_i$ is a normal $F$-subgroup of $G_i^{m_i}$.
%Since $N$ is a normal $F$-subgroup of $\prod_{i=1}^k G_i^{m_i}$, its projection to each factor $G_i$ is normal $F$-subgroup of $G_i$, 
%hence it's either 1 or $G_i$. By Lemma~\ref{L:subprod}, we can write $N$ abstractly as $\prod_{i=1}^k G_i^{n_i}$ and $N_i=G_i^{n_i}$. 
%Then by Lemma~\ref{L:preSchur}, any $F$-automorphism of $N$ maps $N_i$ to itself, so $N_i$ maps to itself under conjugation by $
%\prod_{i=1}^k G_i^{m_i}$ and hence is normal in $\prod_{i=1}^k G_i^{m_i}$. Also from Lemma~\ref{L:preSchur}, we see that for $i\neq j$ 
%the projection $N_i\to G_j^{m_j}$ is trivial and it follows that $N_i$ is an $F$-subgroup of $G_i^{m_i}$.
\end{proof}

The followings are two corollaries of Lemma~\ref{L:subsplit}.
\begin{corollary}\label{C:subwhole}
Let $G_i$ be irreducible $F$-groups for $i=1,\dots, k$ such that for $i\ne j$, we have that $G_i$ and $G_j$ are not isomorphic as $F$-groups. 
Let $N$ be a normal $F$-subgroup of  $\prod_{i=1}^k G_i^{m_i}$.  Then $N=\prod_{i=1}^k G_i^{m_i}$ if and only if 
$\pi_i(N)=G_i^{m_i}$ for each projection $\pi_i : N \ra G_i^{m_i}$.
\end{corollary}

\begin{corollary}\label{C:probsplit}
Let $G_i$ be finite irreducible $F$-groups for $i=1,\dots, k$ such that for $i\ne j$, we have that $G_i$ and $G_j$ are not isomorphic as $F$-groups, and let $m_i$ be non-negative integers.
   Let $R=\prod_{i=1}^k G_i^{m_i}$.
Then
$$
\Prob( [x_1,\dots, x_{n+u} ]_F=R)=\prod_{i=1}^m \Prob( [y_{i,1},\dots, y_{i,n+u} ]_F=G_i^{m_i}),
$$ 
where the $x_k$ are independent, uniform random elements of $R$, and the $y_{i,k}$
are independent, uniform random elements of $G_i^{m_i}.$
\end{corollary}

The next lemma will help us determine when $ [y_{i,1},\dots, y_{i,n+u} ]_F=G_i^{m_i}$.
\begin{lemma}\label{L:critgen}
Let $G$ be an irreducible $F$-group.  If $G$ is non-abelian, then a normal $F$-subgroup $N$ of $G^m$
is all of $G^m$ if and only if it is non-trivial in each of the $m$ projections to $G$.
If $G$ is abelian, then a normal $F$-subgroup $N$ of $G^m$
is all of $G^m$ if and only if the projection onto the product of the first $m-1$ factors is surjective and the projection of $N$
onto the $m$th factor does not factor through the projection onto the product of the first $m-1$ factors.
\end{lemma}
\begin{proof}
The only if direction is clear.
  We let $\pi$ be the projection of $N$ onto the first $m-1$ factors of $G^m$ and  $\pi_m$ the projection onto the last factor.
For the other direction, for non-abelian $G$ we induct and so we have by the inductive hypothesis $\pi(N)=G^{m-1}$.  
For $G$ abelian we have $\pi(N)=G^{m-1}$ as a hypothesis.  
We consider $\pi_m (\ker \pi)$, which must be $1$ or $G$.  If $\pi_m (\ker \pi)$ is $G$, then we see $N=G^m$, as it includes element with every possible first $m-1$ coordinates, and then an element with trivial first $m-1$ coordinates and every possible $m$th coordinate.  
Now we show that we cannot have   $\pi_m (\ker \pi)=1.$
  Suppose for the sake of contradiction that $\pi_m (\ker \pi)=1.$  Then
     since $\ker \pi_m\cap \ker \pi =1$, we have $\ker \pi=1,$ and $\pi$ is an isomorphism on $N$, and in particular $\pi_m$ factors through $\pi$.  So given our hypotheses, this can only happen when $G$ is non-abelian. 
      We write elements $(a,b) \in G^{m-1} \times G$.  Since $\pi_m(N)$ is non-trivial, it must be $G$ by the irreducibility of $G$. For every $b\in G$,  we have some $a\in G^{m-1}$ such that $(a,b)\in N$.
     However, since $N$ is normal, that means $(a,gbg^{-1})\in N$ for every $g\in G$.  Since $\pi_m$ factors through $\pi$, we have that $b=gbg^{-1}$ for every $b,g\in G$, which is a contradiction, since above we saw we can only be in this case if $G$ is non-abelian. 
\end{proof}

Lemma~\ref{L:critgen} lets us compute the probabilities appearing in the right-hand side of Corollary~\ref{C:probsplit} in the following two corollaries.
\begin{corollary}\label{C:nonabprob}
If $G$ is a finite non-abelian irreducible $F$-group, and $y_{k}$ for $k=1,\dots,n+u$
are independent, uniform random elements of $G^{m}$, then
$$
\Prob( [y_{1},\dots, y_{n+u} ]_F=G^{m})=(1-|G|^{-n-u})^m.
$$
\end{corollary}

\begin{corollary}\label{C:abprob}
If $G$ is a finite abelian irreducible $F$-group, and $y_{k}$ for $k=1,\dots,n+u$
are independent, uniform random elements of $G^{m}$, then
$$
\Prob( [y_{1},\dots, y_{n+u} ]_F=G^{m})
=\prod_{k=0}^{m-1} (1-h_F(G)^k| G|^{-n-u}).
$$
\end{corollary}

\begin{proof}
Let $\pi_k$ be the projection of $G^m$ onto the $k$th factor, and $\Pi_k$ the projection of $G^m$ to the first $k$ factors.
We have
\begin{align*}
&\Prob( [y_{1},\dots, y_{n+u} ]_F=G^{m})\\
=&\prod_{k=0}^{m-1} \Prob ( \Pi_{k+1}([y_{1},\dots, y_{n+u} ]_F)=G^{k+1} \,|\, \Pi_{k}([y_{1},\dots, y_{n+u} ]_F)=G^{k} ).
\end{align*}
We condition on the values of $\Pi_{k}(y_i)$, and we still have, with this conditioning, that the $\pi_{k+1}(y_i)$ are uniform, independent random in $G$.  By Lemma~\ref{L:critgen}, given  $\Pi_{k}([y_{1},\dots, y_{n+u} ]_F)=G^{k}$, we will have
$\Pi_{k+1}([y_{1},\dots, y_{n+u} ]_F)=G^{k+1},$ exactly if the map
$\pi_{k+1}|_{[y_{1},\dots, y_{n+u} ]_F}$ does not factor through 
$\Pi_{k}|_{[y_{1},\dots, y_{n+u} ]_F}$.
We have a total of $|G|^{n+u}$ choices for the $(n+u)$-tuple $(\pi_{k+1}(y_1), \cdots, \pi_{k+1}(y_{n+u}))$.  Call 
choice for $(\pi_{k+1}(y_1),\dots, \pi_{k+1}(y_{n+u}))$ \emph{bad}
if $\pi_{k+1}|_{[y_{1},\dots, y_{n+u} ]_F}$ factors through $\Pi_{k}|_{[y_{1},\dots, y_{n+u} ]_F}$.
Since $\Pi_{k}([y_{1},\dots, y_{n+u} ]_F)=G^{k}$, there are $|\Hom_F(G^k,G)|$ choices for
maps from $G^k$ to $G$, each of which gives a bad choice for  $(\pi_{k+1}(y_1),\dots, \pi_{k+1}(y_{n+u}))$ (and all bad choices arise this way). 
For two maps in $\Hom_F(G^k,G) $ to give the same bad choice, they would have to agree on $\Pi_k(y_i)$ for all $i$, and since 
$\Pi_{k}([y_{1},\dots, y_{n+u} ]_F)=G^{k}$, this would imply the two maps in $\Hom_F(G^k,G) $ would be the same.
Thus there are $|\Hom_F(G^k,G)|$ bad choices in $|G|^{n+u}$ for the   $\pi_{k+1}(y_i)$, and as
$|\Hom_F(G^k,G)|=h_F(G)^k$, the corollary follows.
\end{proof}

Theorem~\ref{T:probfrommult} now
 follows from Corollaries~\ref{C:probsplit}, \ref{C:nonabprob}, and \ref{C:abprob}. Also, we can now prove the following lemma which is key for our general approach in Section \ref{S:setup}.
 
\begin{lemma}\label{L:irred}
Let $G$ be a finite group, and  let $N$ be a normal subgroup of $G$.
Let $M$ be the intersection of all maximal proper, $G$-normal subgroups of $N$.
Then $N/M$ is a $G/M$-group under the action of conjugation.
We have that $N/M$ is isomorphic, as an $G/M$-group, to a direct product of irreducible $G/M$-groups. Moreover, among these irreducible $G/M$-groups, the abelian ones all have the action of $G/M$ factor through $G/N$, so are also irreducible $G/N$-groups.
\end{lemma}

\begin{proof}
We consider $N$ as a $G$-group under conjugation.
A subgroup of $N$ is a normal subgroup of $G$ if and only if it is a $G$-subgroup of $N$.  
Taking the quotient modulo $M$ gives us a containment respecting bijection between the $G$-subgroups of $N$ containing $M$ and the $G/M$-subgroups of $N/M$.  Since all maximal proper $G$-subgroups of $N$ contain $M$, the quotient map gives us a bijection between 
the maximal proper $G$-subgroups of $N$ and the maximal proper $G/M$-subgroups of $N/M$, and in particular the quotient $M/M=1$ is the quotient of all the maximal proper $G/M$-subgroups of $N/M$.
Let $M_i$ be the maximal proper $G/M$-subgroups of $N/M$.
 Each $(N/M)/M_i$ is an irreducible $G/M$-group.  We have that $N/M$ is a subgroup of $\prod_i (N/M)/M_i$ that surjects into each factor, $N/M$ is isomorphic to a direct product of irreducible $G/M$-groups by Lemma~\ref{L:subprod}.
On an abelian irreducible $G/M$-group factor, conjugation by any element in $N/M$ gives the trivial group action, so we have the last statement of the lemma.
\end{proof}

%\melanie{Need to say the following somewhere?: A simpler lemma we are using above is if $B$ is a simple $F$-group, and $A$ is an $F$-
%group, and $\Gamma\sub A\times B$ is an $F$-subgroup surjecting onto $A$, then if $\pi_2: \Gamma \ra B$ does not factor through $\pi_1: 
%\Gamma \ra A$ then $\Gamma=A\times B$.  (Proof:$\pi_2(\ker \pi_1)$ is $1$ or $B$.  In the first case, we have factoring through.  In the 
%second case, we have the desired conclusion.)}

\section{Determining factors appearing in $R$}\label{S:factors}

Throughout this section, we assume $S$ is a set of finite groups, $n$ is a positive integer, and $H$ is a finite level $S$ group. If $S$ is finite, then we let
$$1\ra R \ra F\ra H \ra 1$$
be the fundamental short exact sequence associated to $S$, $n$ and $H$ (see Section~\ref{S:setup}). In this section, we will bound   which irreducible $F$-groups are possible factors in $R$. 
A finite irreducible $F$-group is characteristically simple (that is, it contains no proper nontrivial characteristic subgroups) and thus, as a group, a direct product $\Gamma^m$ of isomorphic simple groups.
First, when the group $H$ is fixed, Lemma~\ref{L:HfactR} will bound the possible power $m$ for factors in $R$.    For fixed $S$, Corollary~\ref{C:nonewsimple} will then bound the possible simple group $\Gamma$ for factors in $R$.  We take a slightly longer than necessary  route to Corollary~\ref{C:nonewsimple} because along the way we will develop the technology to prove Corollary~\ref{C:CfpFin}, 
which will later be critical in Section~\ref{S:cntadd} for our proof of countable additivity of $\mu_u$.

\begin{lemma} \label{L:HfactR}
Let $(E, \pi)$ be an $H$-extension such that $G=\ker \pi$ is a finite irreducible $E$-group. Then $G$ is isomorphic to $\Gamma^m$ for some finite simple group $\Gamma$ and $m\leq |H|$.
\end{lemma}

\begin{proof}
Let $G\isom \Gamma^m$, where $\Gamma$ is a finite simple group. If $\Gamma=\Z/p\Z$, then $G \subset \Cen_E(G)$ and  the map $H\ra \Aut(G)=\GL_m(\Z/p\Z)$ defined by conjugation action is an irreducible representation of $H$. Since for any non-zero vector $v \in \F_p^m$, the vectors $hv$, for $h\in H$, span a subrepresentation of $H$, we have the dimension $m$ is at most $|H|$.
 
 If $\Gamma$ is non-abelian, then consider an embedding $ \iota: \Gamma \hookrightarrow \Gamma^m$ such that the image is a normal subgroup. There is an element $a=(a_1, \cdots, a_m) \in \iota(\Gamma)$ such that $a_i$ is not the identity element for some $i$.   
 Let $b\in \Gamma^m$ have $j$th coordinate $1$ for $j\neq i$ and $i$th coordinate $\gamma\in \Gamma$.  Then since $\iota(\Gamma)$ is normal, we have that the commutator $[a,b]\in \iota(\Gamma)$.  The element $[a,b]$ is trivial in all but the $i$th coordinate, where it is $[a_i,\gamma]$.  So the intersection of $\iota(\Gamma)$ and the $i$th factor (which is a normal subgroup of $\Gamma$) contains 
 $[a_i,\gamma]$ for some non-trivial $a_i\in \Gamma$ and all $\gamma\in \Gamma$.  Since $\Gamma$ is a non-abelian simple group, this means 
 the intersection of $\iota(\Gamma)$ and the $i$th factor is non-trivial, and hence all of the $i$th factor.  So   $\iota(\Gamma)$ is exactly the $i$th factor of $\Gamma^m$.  We have thus showed that a normal subgroup of $\Gamma^m$ that is isomorphic to $\Gamma$ must be one of the $m$ factors. So we have a well-defined map $\Aut(\Gamma^m) \to S_m$ (the symmetric group on $m$ elements), and note that $\Inn(\Gamma^m)$ is in the kernel of this map. If $\Gamma^m$ is an irreducible $E$-group, the action of $H$ on the factors must be transitive, which proves $m\leq |H|$.
\end{proof}

\begin{comment}
Since $R$ is a direct product of irreducible $F$-groups, these irreducible $F$-groups must be in the form stated in Lemma~\ref{L:HfactR}.  Next, we want to discuss how the set $S$ decides which factors appear in $R$. We will relate factors in $R$ to the chief factor pairs achieved from $S$ (see Lemma~\ref{L:R-cfp}), and then show in Corollary~\ref{C:new} that the number of factors that can appear in $R$ when $S$ is finite has a bound that is independent of $n$.
%\melanie{need to say independent of $n$ here, and make a more precise statement}
 Counting chief factor pairs will also play an important role in the proof of the countably additivity of $\mu_u$ in Section~\ref{S:cntadd}.
 \end{comment}

Recall that a chief series of a finite group $G$ is a chain of normal subgroups
\begin{equation} \label{eq:chs}
  1=G_0 \lhd G_1 \lhd \cdots \lhd G_r =G
\end{equation}
such that for each $0\leq i \leq r-1$, $G_i$ is normal in $G$ and the quotient group $G_{i+1}/G_i$ is a minimal normal subgroup of $G/G_i$. If $M$ is a minimal normal subgroup of $G$, then define $\rho_M$ to be the homomorphism 
\begin{eqnarray*}
  \rho_M : G & \to & \Aut (M)\\
   g & \mapsto & (x \mapsto g x g^{-1})_{x\in M}.
\end{eqnarray*}
The kernel of $\rho_M$ is the centralizer $\Cen_G(M)$ of $M$ in $G$. So $\rho_M$ gives an isomorphism from $G/\Cen_G(M)$ to the subgroup $\rho_M(G)$ of $\Aut(M)$. In fact, since $M$ is a minimal normal subgroup of $G$, it is a direct product of isomorphic simple groups. If $M$ is a direct product of isomorphic abelian simple groups, i.e. an elementary abelian $p$-group, then $\rho_M(M)=\Inn(M)=1$; otherwise, $\rho_M(M)=\Inn(M)\simeq M$. Thus, $\Inn(M)$ is always a normal subgroup of $\rho_M(G)$ and $\rho_M(G) /\Inn(M) \simeq G/(M\cdot \Cen_G(M))$.

\begin{definition}
A \emph{chief factor pair} is a pair of finite groups $(M,A)$ such that $M$ is an irreducible $A$-group and the $A$-action on $M$ is faithful (hence $A$ is naturally a subgroup of $\Aut(M)$). In particular, the chief series (\ref{eq:chs}) gives a sequence of chief factor pairs $(G_{i+1}/G_i, \rho_{G_{i+1}/G_i}(G/G_i))$, and we call them \emph{chief factor pairs of the series (\ref{eq:chs})}.
\end{definition}

\begin{comment}
If two groups $M_1$ and $M_2$ are isomorphic, then any isomorphism $\alpha: M_1 \to M_2$ induces an isomorphism between the automorphism groups 
\begin{eqnarray}
  \alpha^*: \Aut(M_1) &\to& \Aut(M_2) \label{eq:alphastar}\\
  f & \mapsto& \alpha f \alpha^{-1}.\nonumber
\end{eqnarray}
It's easy to check that $\alpha^*$ maps $\Inn(M_1)$ to $\Inn(M_2)$, so $\alpha^*$ also induces an isomorphism between $\Out(M_1)$ and $\Out(M_2)$. \melanie{Can we cut this paragraph? Do we talk about $\alpha^*$ on Out?}
\end{comment}

\begin{definition}
Two chief factor pairs $(M_1, A_1)$ and $(M_2, A_2)$ are \emph{isomorphic} if there exists an isomorphism $\alpha: M_1 \to M_2$ such that the induced isomorphism $\alpha^*: \Aut(M_1)\to \Aut(M_2)$ maps $A_1$ to $A_2$.
\end{definition}

The following is an analog of the Jordan-H\"older Theorem.

\begin{lemma}\label{L:J-H}
Let $G$ be a finite group. Suppose there are two chief series of $G$:
\begin{eqnarray}
 && 1= G_0 \lhd G_1 \lhd \cdots \lhd G_r =G \label{eq:cs1} \\
 &\text{and}& 1= I_0 \lhd I_1 \lhd \cdots \lhd I_s=G. \label{eq:cs2}
\end{eqnarray}
Then
\begin{enumerate}%[label=(\roman*)]
  \item $r=s$;
  \item the list of isomorphism classes of chief factor pairs $\Big\{\Big(G_{i+1}/G_i , \rho_{G_{i+1}/G_i}(G/G_i) \Big)_{i=0}^{r-1}\Big\}$ is a rearrangement of the list $\Big\{\Big(I_{i+1}/I_i , \rho_{I_{i+1}/I_i}(G/I_i) \Big)_{i=0}^{s-1}\Big\}$.
\end{enumerate}
\end{lemma}

\begin{proof}
We prove this by induction on $|G|$. The case that $|G|=1$ is trivial. Assume the lemma is true for all groups of order less than $k$ and $G$ is a group of order $k$. If $G_1=I_1$ then
\begin{eqnarray*}
  && 1 \lhd G_2/G_1 \lhd \cdots \lhd G_r/G_1 =G/G_1\\
  &\text{and }& 1 \lhd I_2/I_1 \lhd \cdots I_s/I_1 = G/I_1
\end{eqnarray*}
are two chief series of $G/G_1$. So the lemma is proved for $G$ by the induction hypothesis.

Assume $G_1\neq I_1$. Since they are minimal normal subgroups, $G_1\cap I_1 = 1$ and $G_1I_1=G_1\times I_1$. Define $J_2$ to be the product $G_1I_1$. Then $J_2/G_1\simeq I_1$ is a minimal normal subgroup of $G/G_1$ and we can construct a chief series of $G$ passing through $G_1$ and $J_2$
\begin{equation} \label{eq:cs3}
    1\lhd G_1 \lhd J_2 \lhd J_3 \lhd \cdots \lhd J_t=G.
\end{equation}
Comparing chief series (\ref{eq:cs1}) and (\ref{eq:cs3}), it follows by the inductive hypothesis for the group $G/G_1$ that $r=t$ and 
  \begin{eqnarray}
    \left\{ \Big( G_{i+1}/G_i , \rho_{G_{i+1}/G_i} (G/G_i)\right)_{i=0}^{r-1}\Big\} &\sim& \Big\{ \Big( G_1, \rho_{G_1}(G)\Big), \Big(J_2/G_1, \rho_{J_2/G_1} (G/G_1)\Big), \label{eq:list} \\
    & &\Big(J_{i+1}/J_i, \rho_{J_{i+1}/J_i}(G/J_i)\Big)_{i=2}^{t-1} \Big\} \nonumber
  \end{eqnarray}
  where the symbol $\sim$ means ``is a rearrangement of''. Let $\pi$ be the quotient map $G\to G/G_1$. As $G_1\unlhd \Cen_G(I_1)$, if an element in $G$ centralizes $I_1$, then its image under $\pi$ centralizes $\pi(I_1)=J_2/G_1$. It follows that $\pi(\Cen_G(I_1))\subseteq \Cen_{G/G_1}(J_2/G_1)$. Conversely, if $a$ is an element in $G$ such that $\pi(a)\in \Cen_{G/G_1}(J_2/G_1)$, then for every $h \in I_1$, we have $\pi(aha^{-1})=\pi(h)$, which indicates that there exists $g\in G_1$ such that $aha^{-1}=hg$. But $I_1\unlhd G$, so $aha^{-1}\in I_1$. It follows from $I_1\cap G_1=1$ that $g=1$, and hence $a\in \Cen_G(I_1)$, which proves $\pi(\Cen_G(I_1))=\Cen_{G/G_1}(J_2/G_1)$. Thus we have
  \begin{eqnarray*}
    \faktor{G}{\Cen_G(I_1)}&\cong& \faktor{G/G_1}{\Cen_G(I_1)/G_1}\\
    &\cong& \faktor{G/G_1}{\Cen_{G/G_1}(J_2/G_1)}.
  \end{eqnarray*}
  Therefore the chief factor pairs $\Big(I_1, \rho_{I_1}(G)\Big)$ and $\Big(J_2/G_1, \rho_{J_2/G_1}(G/G_1)\Big)$ are isomorphic. So the list (\ref{eq:list}) is 
  \begin{equation} \label{list1}
    \sim \Big\{\Big(G_1, \rho_{G_1}(G)\Big), \Big(I_1, \rho_{I_1}(G) \Big), \Big(J_{i+1}/J_i, \rho_{J_{i+1}/J_i}(G/J_i)\Big)_{i=2}^{t-1}\Big\}.
  \end{equation}
   Similarly, by comparing the following chief series of $G$
  \begin{equation} \label{cs4}
    1\lhd I_1 \lhd J_2 \lhd J_3 \lhd \cdots \lhd J_t=G.
  \end{equation}
  with (\ref{eq:cs2}), we finish the proof of the lemma.
\end{proof}

\begin{definition}
 If $G$ is a finite group, then define $\CF(G)$ to be the set consisting of all isomorphism classes of chief factor pairs of a chief series of $G$ ($\CF(G)$ does not depend on the choice of chief factor series by Lemma~\ref{L:J-H}). If $T$ is a set of finite groups, then $$\CF(T):=\bigcup\limits_{G\in T} \CF(G).$$
\end{definition}

The following lemma shows that every factor in $R$ comes from $\CF(\bS)$.

\begin{lemma}\label{L:R-cfp}
Let $S$ be a finite set of finite groups, and $R$, $F$ and $H$ as defined at the beginning of this section. If $G$ is an irreducible $F$-subgroup of $R$, then $(G, \rho_G(F))\in \CF(\bS)$ and $\rho_G(F)/\Inn(G)$ is isomorphic to a quotient of $H$.
\end{lemma}

\begin{proof}
Since $F$ is a finite level $S$ group and $G$ is a minimal normal subgroup of $F$, we have $(G, F/\Cen_F(G)) \in \CF(\bS)$. Further, $R$ is a direct product of irreducible $F$-groups, so $R$ is contained in $\Cen_F(G) \cdot G$ and it follows that $(F/\Cen_F(G))/\Inn(G) = F/(\Cen_F(G)\cdot G)$ is a quotient of $H$.
\end{proof}

In the rest of this section, we will bound the size of chief factor pairs.

\begin{lemma}\label{L:CFbarS}
If $S$ is a set of finite groups that is closed under taking subgroups and quotients, then $\CF(\bS)=\CF(S)$.
\end{lemma}

\begin{proof}
Since $\bS$ is the closure of $S$ under taking finite direct products, quotients and subgroups, it suffices to show that none of these three actions creates new chief factor pairs not belonging to $\CF(S)$.

First, taking direct products and quotients does not create new chief factor pairs. If $G$ and $J$ are finite groups with chief series $1\lhd G_1 \lhd \cdots \lhd G_r = G$ and $1\lhd J_1 \lhd \cdots \lhd J_s = J$. Then the following chief series of $G\times J$
$$1 \lhd G_1 \times 1 \lhd \cdots G\times 1 \lhd G\times J_1 \lhd \cdots \lhd G\times J$$
implies that $\CF(G\times J)=\CF(G)\cup \CF(J)$. If $N$ is a normal subgroup of $G$, then $\CF(G/N)\subseteq \CF(G)$ since we can always find a chief series of $G$ passing through $G/N$.

Finally, assume $J$ is a subgroup of $G$ for $G\in \bS$ such that $\CF(G)\subseteq \CF(S)$. We want to prove $\CF(J)\subseteq \CF(S)$. Let $1 \lhd G_1 \lhd \cdots \lhd G_r = G$ be a chief series of $G$. We can construct a chief series of $J$ that passes through $G_i \cap J$ for every $i=1, \cdots, r$. The chief factor pairs achieved from the elements between $G_i \cap J$ and $G_{i+1}\cap J$ are achieved from the group $J/(G_i\cap J)\simeq (J\cdot G_i)/ G_i$, which is a subgroup of $G/G_i$. Thus it's enough to consider the positions between $1$ and $G_1\cap J$. Since $(G_1, \rho_{G_1}(G))\in \CF(G) \subseteq \CF(S)$, there is a group $G'\in S$ and a minimal subgroup $G'_1$ of $G'$ such that the chief factors $(G_1, \rho_{G_1}(G))$ and $(G'_1, \rho_{G'_1}(G'))$ are isomorphic, i.e. $\exists$ $\alpha: G_1\overset{\sim}{\to}G'_1$ such that $\alpha^*: \Aut(G_1)\overset{\sim}{\to} \Aut(G'_1)$ maps $\rho_{G_1}(G)$ to $\rho_{G'_1}(G')$. Define $A:=\rho_{G_1}(J)=(J\cdot \Cen_G(G_1))/\Cen_G(G_1)$ that is a subgroup of $\rho_{G_1}(G)$. Note that the action of $A$ on $G_1$ actually stabilizes $G_1\cap J$. Let $J':=\rho_{G'_1} ^{-1}(\alpha^*(A))$ and $J'_1=\alpha(G_1\cap J)$. So $J'$ is a subgroup of $G'$ satisfying the following short exact sequence 
$$1 \to \Cen_{G'}(G'_1) \to J' \to \alpha^*(A)\to 1.$$
and $J'_1$ is a subgroup of $G'_1\cap J'$. 

Since $\Cen_{G'}(G_1')\leq \Cen_{J'}(G'_1\cap J')$, the action of $J'$ via conjugation on $G'_1\cap J'$ factors through $\alpha^*(A)$. Also, since the $\alpha^*(A)$ action on $G'_1$ stabilizes $J'_1$, we have that $J'_1$ is a normal subgroup of $J'$. Because $G_1\cap J$ with the action of $A$ is isomorphic to $J'_1$ with the action of $\alpha^*(A)$, every chief factor pair of $G$ achieved from positions between 1 and $G_1\cap J$ is also a chief factor pair of $J'$ achieved via a series passing through $J'_1$. Finally, $J'$ as a subgroup of $G'$ belongs to $S$, so $\CF(J)\subseteq \CF(S)$ and we prove the lemma.

\end{proof}

\begin{corollary}\label{C:nonewsimple}
If $S$ is a set of finite groups, and $\Gamma \in \bar{S}$ is a simple group, then $\Gamma$ is in the closure of $S$ under taking subgroups and quotients.
\end{corollary}

\begin{proof}
If $\Gamma\in \bS$ is a simple group, then $(\Gamma, \Inn(\Gamma))\in \CF(\bS)$. By Lemma~\ref{L:CFbarS}, $\Gamma$ is in the closure of $S$ under taking subgroups and quotients.
\end{proof}

\begin{corollary}\label{C:CfpFin}
Let $S$ be a finite set of finite groups. Then $\CF(\bS)$ is a finite set. Moreover, if $\ell$ is the upper bound of the orders of groups in $S$, then for any pair $(M,A)\in \CF(\bS)$, the quotient $A/\Inn(M)$ is of level $\ell-1$.  
\end{corollary}

\begin{proof}
Without lost of generality, let's assume $S$ is closed under taking subgroups and quotients. By Lemma~\ref{L:CFbarS}, $\CF(\bS)=\CF(S)$ is finite, and for any chief factor pair $(M,A)\in \CF(\bS)$, there is a group $G\in S$ such that $(M,A)\in \CF(G)$. If $M$ is abelian, then $|M||A|\leq |G| \leq \ell$; otherwise, $M$ is non-abelian and $|A|=|M||A/\Inn(M)|\leq |G| \leq \ell$. In either case, we have $|A/\Inn(M)|\leq\frac{\ell}{2}\leq \ell-1$.
\end{proof}

\begin{remark}
	The statement in Corollary~\ref{C:CfpFin} remains true if $\ell-1$ is replaced by $\lfloor \ell/2 \rfloor$ but we will not use that stronger statement.
\end{remark}

\section{Counting maximal quotients of irreducible $F$-groups}\label{S:countqt}

In order to apply Theorem~\ref{T:probfrommult} to a group that we know, abstractly,  to be a product of irreducible $F$-groups, we need to know the multiplicities of the various irreducible $F$-groups in the product.  In this section, we relate those multiplicities to a count of surjections.

\begin{theorem}\label{T:countRsub}
Let $G_i$ be finite irreducible $F$-groups for $i=1,\dots, k$ such that $G_i$ and $G_j$ are not isomorphic for $i\ne j$. Then if $G_j$ is abelian
$$
\# \Sur_F\left(\prod_{i=1}^k G_i^{m_i} , G_j \right) =h_F(G_j)^{m_j}-1
$$
and if $G_j$ is non-abelian
$$
\# \Sur_F\left(\prod_{i=1}^k G_i^{m_i} , G_j \right) = m_j|\Aut_F(G_j)|.
$$
\end{theorem}

This theorem follows immediately from Lemma~\ref{L:preSchur} and the following lemmas.

\begin{lemma}\label{L:justj}
Let $G_i$ be finite irreducible $F$-groups for $i=1,\dots, k$ such that for $i\ne j$, we have that $G_i$ and $G_j$ are not isomorphic.  The  restriction map
$$
\Sur_F\left(\prod_{i=1}^k G_i^{m_i} , G_j \right) \ra \Hom_F\left( G_j^{m_j} , G_j \right)
$$
is a bijection to $\Sur_F\left( G_j^{m_j} , G_j \right)\sub \Hom_F\left( G_j^{m_j} , G_j \right)$.
\end{lemma}

\begin{proof}
Note that in a surjection, each $G_i$ must go to a normal subgroup of $G_j$, and so by Lemma~\ref{L:preSchur} the restriction to every $G_i$ factor for $i\ne j$ is trivial.  So that proves the above restriction map is injective.  
The restriction map is surjective to $\Sur_F\left( G_j^{m_j} , G_j \right)$ since $G_j^{m_j}$ is a quotient of $\prod_{i=1}^k G_i^{m_i}$. 
\end{proof}

\begin{lemma}\label{L:morto1}
Let $G$ be a finite irreducible $F$-group and $m$ a positive integer.  
We have
$$
\Hom_F(G^m, G) \sub \Hom_F(G,G)^m
$$
by restriction to each factor.
If $G$ is abelian, then this inclusion is an equality.  If $G$ is non-abelian, then we have that $\Hom_F(G^m, G) $ is the subset of the
$m$-tuples $\Hom_F(G,G)^m$ where at most $1$ coordinate is a non-trivial morphism in $\Hom_F(G,G)$.
The only homomorphism that is not surjective among those above is the trivial morphism.  
\end{lemma}

\begin{proof}
If $G$ is abelian, then for $\phi_i\in \Hom_F(G,G)$, we have a morphism $\phi: G^m \ra G$ such that $\phi(a_1,\dots,a_m)=\prod_{i=1}^m \phi_i(a_i)$.
Note for $\phi\in \Hom_F(G^m, G)$, with restrictions $\phi_i$ to the factors, we have that $a\in\phi_i(G)$ and $b\in\phi_j(G)$ commute for $i\neq j$.  Since $\phi_i(G)$ is $1$ or $G$, if $G$ is non-abelian we see that at most one $\phi_i$ can be non-trivial.  Moreover, clearly the $m$-tuples $\Hom_F(G,G)^m$ where at most $1$ coordinate is a non-trivial morphism in $\Hom_F(G,G)$  give elements of $\Hom_F(G^m,G)$.  
For an $F$-morphism $G\ra G$, if it is non-trivial, it must be injective (since its kernel is a normal $F$-subgroup), and thus surjective.
\end{proof}

\section{Determination of $\mu_{u,n}$ on basic open sets}\label{S:basics}

Recall that $X_{u,n}$ is the random group obtained by taking quotient of $\hat{F}_n$ by $n+u$ independent random relations from Haar measure. The goal of this section is to prove Theorem~\ref{T:calc}, in which we will give $\Prob((X_{u,n})^{\bS}\isom H)$ for every finite set $S$ and finite level $S$ group $H$, i.e. determine the measures of the basic open sets in the distributions coming from our random groups. 
Throughout this section, we assume $n\geq 1$, $S$ is a finite set of finite groups, $H$ is a finite level $S$ group and 
$$1\ra R \ra F \ra H \ra 1$$
is the fundamental short exact sequence associated to $S$, $n$ and $H$. 
For any abelian irreducible $H$-group $G$, we define $m(S,n,H,G)$ to be the multiplicity of $G$ in $R$ as an $H$-group under conjugation (see Lemma~\ref{L:irred}). Let $G$ be a non-abelian finite group. Let $G_i$ be the irreducible $F$-group structures one can put on $G$. Then we define $m(S,n,H,G)$ to be the sum (over $i$) of the multiplicity of the $G_i$ in $R$ as an $F$-group under conjugation.
Equation~\eqref{E:setup} and Theorem~\ref{T:probfrommult} allow us to express $\Prob((X_{u,n})^{\bS}\isom H)$ in terms of the multiplicities
$m(S,n,H,G)$.  The work of this section will be to find
 explicit formulas for these $m(S,n,H,G)$ (given in Corollaries~\ref{C:finalmab} and \ref{C:finalmnon}).

\begin{theorem}\label{T:calc}
Let $S$ be a finite set of finite groups
 and $H$ a finite level $S$ group. Let $n\geq 1$ and $u> -n$ be integers.
Then
\begin{eqnarray}\label{E:fixedn}
&&\Prob((X_{u,n})^{\bS}\isom H)\\
&=&\frac{|\Sur(\hat{F}_n,H)|}{|\Aut(H)||H|^{n+u}}\prod_{\substack{G \in \mathcal{A}_H}} 
\prod_{k=0}^{m(S,n,H,G)-1} (1-\frac{h_H(G)^k}{ |G|^{{n+u}}})
 \prod_{\substack{ G\in \mathcal{N}}}  (1-|G|^{-{n-u}})^{m(S,n,H,G)},\nonumber
\end{eqnarray}
and we have
\begin{eqnarray}\label{E:lim}
&&
\lim_{n\ra\infty} \Prob((X_{u,n})^{\bS}\isom H)\\
&=&\frac{1}{|\Aut(H)||H|^{u}}\prod_{\substack{G \in \mathcal{A}_H}} 
\prod_{i=1}^{\infty} (1-\lambda(S,H,G)
\frac{h_H(G)^{-i}}{ |G|^{{u}}})
 \prod_{\substack{ G\in \mathcal{N}}}  e^{-|G|^{-u}\lambda(S,H,G)}.\nonumber
\end{eqnarray}

Further, if $G\in \mathcal{A}_H\cup \mathcal{N}$ is isomorphic as a group to $\Gamma^j$ for some simple group $\Gamma$, and either 1) $\Gamma$ is not in the closure of $S$ under taking subgroups and quotients, or 2) $j>|H|$, then
$m(S,n,H,G)=\lambda(S,H,G)=0$.

\end{theorem}

\begin{remark}\label{R:finprod}
%The non-trivial terms in the product over $\mathcal{A}_H$ and $\mathcal{N}$ are those with positive $m(S,n,H,G)$. 
The products over $\mathcal{A}_H$ and  
$\mathcal{N}$ appearing in Theorem~\ref{T:calc} are actually finite products (except for trivial terms), because of the last statement in the theorem. %and Lemma~\ref{L:nonewsimple}, which implies there are only finitely many simple groups in $\bar{S}$.
%By Lemma~\ref{L:R-cfp}, they all come from the first components of chief factor pairs in $\CF(\bS)$. 
%Therefore, this product is actually a finite product by Lemma~\ref{L:CfpFin}.
\end{remark}

\begin{remark}
We will show in Section~\ref{S:arbitraryS} that statement of Theorem~\ref{T:calc} also works for an arbitrary set $S$ of finite groups.
\end{remark}

First, we need to define the M\"obius function on a poset of $H$-extensions.  Given a finite group $H$, there is a poset $\mathcal{E}_H$ of $H$-extensions (\emph{not} isomorphism classes of $H$-extensions) where $(E,\pi) \leq (E',\pi')$ if $(E,\pi)$ is a sub-$H$-extension of $(E',\pi')$.  (This relation is defined for literal sub-$H$-extensions and not  $H$-extensions just isomorphic to a subextension.)
We let $\nu(D,E)$ be the M\"obius function of this poset (we drop the maps to $H$ in the notation but they are implicit) so that for two $H$-extensions $D$ and $E$ we have
\begin{align*}
\nu(E,E)&=1\\
\nu(D,E)&=-\sum_{\substack{D'\in \mathcal{E}_H \\D< D' \leq E }} \nu(D',E) \quad\quad \textrm{    if $D\ne E$}
\end{align*}
so that in particular
\begin{align*}
\nu(D,E)&=0 \textrm{   if $D$ is not a sub-$H$-extension of $E$}\\
\sum_{\substack{D'\in \mathcal{E}_H \\D\leq D' \leq E }} \nu(D',E)&= 
\begin{cases}
1 &\textrm{if $D=E$}\\
0 & \textrm{otherwise}.
\end{cases}
\end{align*}

Theorem~\ref{T:countRsub} relates our key multiplicities $m(S,n,H,G)$ to the number of $F$-surjections from $R$ to $G$.  An $F$-surjection $R\ra G$ has a kernel $K$, and we have a surjection from our $H$-extension $(\hat{F}_n)^{\bS}\ra H$ to the $H$-extension $F/K \ra H$.  The next proposition will count such surjections of $H$-extensions.

\begin{proposition}\label{P:countfromF}
Let $n\geq1$ be an integer,  $S$ a finite set of finite groups, and $H$ a finite level $S$ group.
Let  $(\hF_n)^{\bS}\stackrel{\rho}{\ra}H$ be an $H$-extension structure on $(\hF_n)^{\bS}$.
Let $E \stackrel{\pi}{\ra} H $ be a finite $H$ extension.  We have
$$
|\Sur(\rho,\pi)|=\begin{cases}
\sum_{\substack{D\in\mathcal{E}_H, D\leq E}} \nu(D,E) \left( \frac{|D|}{|H|} \right)^n &\textrm{   if $E$ is level $S$}
\\
0  &\textrm{   otherwise}.
\end{cases}
$$
\end{proposition}

\begin{proof}
If $(\hF_n)^{\bS} \ra E$ is a surjection, then $E$ is level $S$ and hence $|\Sur(\rho,\pi)|=0$ if $E$ is not level $S$. 
If $E$ is level $S$ and $(D,\psi)\leq (E,\pi)$, surjections $(\hF_n)^{\bS} \ra D$ exactly correspond to surjections $\hF_n \ra D$, i.e. choices of image for each generator $x_1,\dots,x_n$ of $\hat{F}_n$ such that their images generate $D$. 
For each generator $x_i$ of $\hF_n$, we have a fixed coset of $\ker(\pi)$ in $D$ it can land in to actually obtain a surjection compatible with the maps to $H$. 
The number of homomorphisms $\hF_n \ra D$ where the generators go to the appropriate cosets is $(|D|/|H|)^n$.  
Let $E'$ be a subgroup of $D$ that could be generated by some $y_1,\dots,y_n$ with each $y_i$ in the required cosets of $\ker(\pi)$.  Since $\rho$ is a surjection, it follows that $\pi(E')=H$.
So we have
$$
\left(\frac{|D|}{|H|}\right)^n=\sum_{\substack{(E',\phi)\in\mathcal{E}_H\\ (E',\phi)\leq (D,\psi)}} |\Sur_H(\rho,\phi)|.
$$
Using M\"obius inversion, we obtain the result. We can sum the above as follows.  Given a finite $H$-extension $(E,\pi)$ of level $S$, we have 
\begin{align*}
\sum_{(D,\psi) \leq (E,\pi)} \nu(D,E) \left(\frac{|D|}{|H|}\right)^n&=\sum_{(D,\psi) \leq (E,\pi)} \nu(D,E) \sum_{\substack{
%(D,\phi)\in\mathcal{E}_H\\
 (E',\phi)\leq (D,\psi)}}   |\Sur_H(\rho,\phi)|\\
&=\sum_{(E',\phi) \leq (E,\pi)}   |\Sur_H(\rho,\phi)| \sum_{ (E',\phi)\leq (D,\psi) \leq (E,\pi)} \nu(D,E) \\
&=  |\Sur_H(\rho,\pi)|,
\end{align*}
as desired.
\end{proof}

%\begin{remark}\label{R:universal}
%Note that in Proposition~\ref{P:countfromF}, the number of surjections does not depend on the choice of map $(\hF_n)^{\bS}{\ra}H$. In %particular, if we let $E$ to be $(\hat{F}_n)^{\bS}$, then Proposition~\ref{P:countfromF} implies that $|\Sur(\rho, \pi)|=|\Sur(\pi,\pi)|%>0$ and hence any two surjections $\pi, \rho: (\hat{F}_n)^{\bS} \to H$ are isomorphic as $H$-extensions.
%\end{remark}

Now we will build on Proposition~\ref{P:countfromF} to find $|\Sur_F(R,G)|$, after which we can then use Theorem~\ref{T:countRsub} to find the $m(S,n,H,G)$.  We will first do the case of abelian $G$, and then non-abelian $G$.

\begin{proposition}[Counting surjections from $R$ to abelian $G$]\label{P:countNother}
Let $H$, $F$, and $R$ be defined as at the beginning of this section.
Let $G$ be an abelian irreducible $F$-group.
Then 
\begin{align*}
|\Sur_F(R,G)|=|\Aut_F(G)|\sum_{\substack{\textrm{isom. classes of $H$-extensions $(E,\pi)$}\\ \textrm{$\ker \pi \simeq G$ as an $H$-group} \\
\textrm{$E$ is level $S$} }}
\frac{\sum_{\substack{D\in\mathcal{E}_H, D\leq E}} \nu(D,E) \left( \frac{|D|}{|H|} \right)^n}{|\Aut_H(E,\pi)|}.
\end{align*}
if the action of $F$ on $G$ factors through $F\ra H$ (i.e. elements of $R$ act trivially on $G$) and $|\Sur_F(R,G)|=0$ otherwise.
\end{proposition}

\begin{proof}
We have that $|\Sur_F(R,G)|$ is $|\Aut_F(G)|$ times
the number of $F$-subgroups $M$ of $R$ such
that $R/M$ under $F$-conjugation is isomorphic to $G$ as an $F$-group.
If $M$ is an $F$-subgroup of $R$ such that $R/M$ is abelian, then the action of $F$ via conjugation on $R/M$ factors through $H$ (because conjugation by elements from $R$ is trivial in $R/M$ as $R/M$ is abelian).
So suppose that the action of $F$ on $G$ factors through $H$.
We have
the number of $F$-subgroups $M$ of $R$ such
that $R/M$ is isomorphic to $G$ as an $F$-group is
\begin{eqnarray*}
&&\sum_{\substack{\textrm{isom. classes of} \\
\textrm{$H$-extensions $(E,\pi)$}}}
\#\left\{ F\textrm{-subgroups $M$ of $R$}\,  \Bigg|\,
\begin{aligned}
 & (F/M\ra H)\isom (E,\pi) \textrm{ as $H$-exts,}\\
 & \textrm{$R/M\isom G$ as $F$-groups } 
 \end{aligned} \right\}\\
&=&\sum_{\substack{\textrm{isom. classes of} \\
\textrm{$H$-extensions $(E,\pi)$}\\
\ker \pi\isom G \textrm{ as groups}
}}
\#\left\{ F\textrm{-subgroups $M$ of $R$}\,  \Bigg|\,
\begin{aligned}
 & (F/M\ra H)\isom (E,\pi) \textrm{ as $H$-exts,}\\
 & \textrm{$R/M\isom G$ as $F$-groups } 
 \end{aligned} \right\}.
\end{eqnarray*}
Given that $R/M$ is abelian (which is guaranteed by the group isomorphism $\ker \pi\isom G$ and the $H$-extension isomorphism $(F/M\ra H)\isom (E,\pi)$), since the action of $F$ on $R/M$ factors through $H$, we have that 
$R/M$ is isomorphic to $G$ as an $F$-group if and only if it is isomorphic to $G$ as an $H$-group.
Given $(F/M\ra H)\isom (E,\pi)$, this is the same as requiring $\ker \pi\isom G$ as an $H$-group.  
Thus the above sum is equal to 

\begin{eqnarray*}
  && \sum_{\substack{\text{isom. classes of $H$-extension }(E,\pi) \\ \ker \pi \simeq G \text{ as an $H$-group}}} \#\left\{ F\text{-subgroups $M$ of $R$}\mid (F/M \ra H) \simeq (E,\pi) \text{ as $H$-exts} \right\}\\
  &=& \sum_{\substack{\text{isom. classes of $H$-extension }(E,\pi) \\ \ker \pi \simeq G \text{ as an $H$-group}}} \frac{ \# \left\{ (M, \phi)\,\Bigg|\,
  \begin{aligned}
    & M \text{ an $F$-subgroup of }R\\
    & \phi: (F/M \ra H) \simeq (E, \pi)
  \end{aligned}
  \right\}}{|\Aut_H(E,\pi)|}.
\end{eqnarray*}
Note that the data $(M,\phi)$ above is exactly the same as the data of a surjection of $H$-extensions from $F\ra H$ to $E\ra H$.

Now let $(E,\pi)$ be an $H$-extension with $\ker \pi$ (via conjugation) an abelian irreducible $H$-group. Consider a surjection of $H$-extensions from $(\hat{F}_n)^{\bS}\ra H$ to $E\ra H$, in which the map  $(\hat{F}_n)^{\bS}\ra E$ has kernel $K$. Let $N$ denote the kernel of $(\hat{F}_n)^{\bS}\to H$. Then $N/K \simeq \ker \pi$ is an irreducible $(\hat{F}_n)^{\bS}$-group, and so $K$ is a maximal proper $(\hat{F}_n)^{\bS}$-subgroup of $N$.  So the map $(\hat{F}_n)^{\bS}\ra E$ factors through $F$.  On the other hand, any surjection of $H$-extensions from $F\ra H$ to $E\ra H$ clearly extends to a surjection of $H$-extensions from $(\hat{F}_n)^{\bS}\ra H$ to $E\ra H$.   Thus the above sum is equal to
  \begin{align*}
\sum_{\substack{\textrm{isom. classes of $H$-extensions $(E,\pi)$}\\ \textrm{$\ker \pi \simeq G$ as an $H$-group}}}
\frac{|\Sur_H((\hat{F}_n)^{\bS}\ra H,\pi)|}{|\Aut_H(E,\pi)|}.
\end{align*}
 The result now follows from applying Proposition~\ref{P:countfromF} above, after dividing out by the number of choices of isomorphism to $(E,\pi)$.
\end{proof}

We now can determine the multiplicities of the abelian irreducible $F$-groups in $R$ by combining Theorem~\ref{T:countRsub} and Proposition~\ref{P:countNother}.
\begin{corollary}[Multiplicities of abelian $G$ in $R$]\label{C:finalmab}
Let $H$, $F$ and $R$ be as above. Let $G$ be an abelian irreducible $H$-group.
%If the action of $F$ on $G$ does not factor through the map $F\ra H$, then $m(S,n,H,G)=0$.
%If the action of $F$ on $G$  does factor through the map $F\ra H$, 
Then
$$
\frac{h_H(G)^{m(S,n,H,G)}-1}{h_H(G)-1} =\sum_{\substack{\textrm{isom. classes of $H$-extensions $(E,\pi)$}\\ \textrm{$\ker \pi \simeq G$ as an $H$-group} \\
\textrm{$E$ is level $S$} }}
\frac{\sum_{\substack{D\in\mathcal{E}_H, D\leq E}} \nu(D,E) \left( \frac{|D|}{|H|} \right)^n}{|\Aut_H(E,\pi)|}.
$$
%In particular, we have that $m(S,n,H,G)$ does not depend on the choice of surjection $(\hF_n)^{\bS} \ra H$.
\begin{comment}
 and
\begin{align*}
\lim_{n\ra \infty} \frac{h_H(G)^{m'(S,n,H,G)}}{|G|^n}% &=\lim_{n\ra \infty} (h(G)-1) 
%\sum_{\substack{\textrm{isom. classes of $H$-extensions $(E,\pi)$}\\ \textrm{$\ker \pi$ isom. $G$ as an $H$-%group} \\
%\textrm{$E$ quasi-level $S$}  }}
%\frac{\sum_{\substack{D \textrm{ subgroup of } E\\\pi(S)=H}} \nu(S,E) \left( \frac{|D|}{|E|} \right)^n}{\#
%\Aut_H(E,\pi)}\\
&= (h_H(G)-1) 
\sum_{\substack{\textrm{isom. classes of $H$-extensions $(E,\pi)$}\\ \textrm{$\ker \pi \simeq G$ as an $H$-group} \\
\textrm{$E$ level $S$}  }}
\frac{1}{|\Aut_H(E,\pi)|}.
\end{align*}
\end{comment}
\end{corollary}

Next, we will apply a similar plan to obtain the multiplicities of the non-abelian $G$, but there is an important difference from the abelian case.
When $\ker ( E \stackrel{\pi}{\ra} H)$ is non-abelian, a surjection of $H$-extensions $F\ra E$ still gives an $F$-group structure on $\ker \pi$ by conjugation in $E$, but, unlike in the case when $\ker\pi$ is abelian, that $F$-group structure is not necessarily determined by the isomorphism type of the $H$-extension $(E,\pi).$
So in this case it is most convenient to add together, for each possible \emph{underlying group} $G$ of a non-abelian irreducible $F$-group, all surjections $F\ra E$ over all $G$ extensions $E$ of $H$.

\begin{proposition}[Counting surjections from $R$ to non-abelian $G$]\label{P:mulnonab}
Let $H$, $F$ and $R$ be as above.
%Let $H$ be a quasi-level $S$ group, and $F_{g,S}\ra H$ a surjection.  The notation for $R$ and $N$ is above.
Let $G$ be a finite non-abelian group. 
% It may have have several non-isomorphic structures as an irreducible $F$ group or not at all.  
Let $G_i$ be the pairwise non-isomorphic irreducible $F$-group structures on $G$ for $1\leq i \leq k$ ($k$ may be $0$).
Then
\begin{align*}
\sum_{i=1}^k
 \frac{|\Sur_F(R,G_i)|}{|\Aut_F(G_i)|}=
&\sum_{\substack{\textrm{isom. classes of $H$-extensions $(E,\pi)$}\\ \textrm{$\ker \pi \simeq G$ }
\\ \textrm{$\ker \pi$ irred. $E$-group} \\
\textrm{$E$ is level $S$} }} 
\frac{\sum_{\substack{D\in\mathcal{E}_H, D\leq E}} \nu(D,E) \left( \frac{|D|}{|H|} \right)^n}{|\Aut_H(E,\pi)|}.
\end{align*}
\end{proposition}

\begin{proof}
We note that ${|\Sur_F(R,G_i)|}/{|\Aut_F(G_i)|}$ is the number of $F$-subgroups of $R$ whose corresponding quotient is isomorphic to $G_i$ as an $F$-group.
We have
\begin{eqnarray*}
&&\sum_{i=1}^k
 \frac{|\Sur_F(R,G_i)|}{|\Aut_F(G_i)|}\\
&=&\sum_{i=1}^{k}
\sum_{\substack{\textrm{isom. classes of}\\ \textrm{ $H$-extensions $(E,\pi)$}
 }}
\#\left\{\textrm{$F$-subgroups $M$ of $R$} \,\Bigg|\,
\begin{aligned}
&(F/M\ra H)\isom (E,\pi) \text{ as $H$-exts}\\
& R/M\isom G_i  \text{ as $F$-groups }
\end{aligned}\right\}\\
&=&\sum_{i=1}^{k}
\sum_{\substack{\textrm{isom. classes of}\\\textrm{ $H$-extensions $(E,\pi)$}\\\textrm{$\ker \pi \simeq G$ }
 }}
\#\left\{\textrm{$F$-subgroups $M$ of $R$} \,\Bigg|\,
\begin{aligned}
&(F/M\ra H)\isom (E,\pi) \text{ as $H$-exts}\\
& R/M\isom G_i  \text{ as $F$-groups }
\end{aligned}\right\}\\
%%%
&=&
\sum_{\substack{\textrm{isom. classes of}\\\textrm{ $H$-extensions $(E,\pi)$}\\\textrm{$\ker \pi \simeq G$ }
 }} \sum_{i=1}^{k}
\#\left\{\textrm{$F$-subgroups $M$ of $R$} \,\Bigg|\,
\begin{aligned}
&(F/M\ra H)\isom (E,\pi) \text{ as $H$-exts}\\
& R/M\isom G_i  \text{ as $F$-groups }
\end{aligned}\right\} \\
%%%
&=&
\sum_{\substack{\textrm{isom. classes of}\\\textrm{ $H$-extensions $(E,\pi)$}\\\textrm{$\ker \pi \simeq G$ }
 \\ \textrm{$\ker \pi$ irred. $E$-group} 
 }} \sum_{i=1}^{k}
\#\left\{\textrm{$F$-subgroups $M$ of $R$} \,\Bigg|\,
\begin{aligned}
&(F/M\ra H)\isom (E,\pi) \text{ as $H$-exts}\\
& R/M\isom G_i  \text{ as $F$-groups }
\end{aligned}\right\} \\
%%%
%%%
&=&\sum_{\substack{\textrm{isom. classes of $H$-extensions $(E,\pi)$}\\ \textrm{$\ker \pi \simeq G$ }
\\ \textrm{$\ker \pi$ irred. $E$-group} 
 }}
\#\{\textrm{$F$-subgroups $M$ of $R \,|\,(F/M\ra H)\isom (E,\pi)$ as $H$-exts  }\}\\
%%%
%%%
&=&\sum_{\substack{\textrm{isom. classes of $H$-extensions $(E,\pi)$}\\ \textrm{$\ker \pi \simeq G$ }
\\ \textrm{$\ker \pi$ irred. $E$-group} 
 }}\frac{
\#\left\{(M,\phi)\,\Bigg|\,
\begin{aligned}
&\textrm{$M$ an $F$-subgroup of $R$} \\
&\textrm{$\phi:(F/M\ra H)\isom (E,\pi)$ as $H$-exts}
\end{aligned}
\right\}}{|\Aut_H(E,\pi)|}.
\end{eqnarray*}
The second equality follows because $(F/M\ra H) \isom (E,\pi)$ and $R/M\isom G_i$ imply that $\ker \pi \isom G$.
The fourth equality follows because $(F/M\ra H) \isom (E,\pi)$ and $R/M$ being an irreducible $F$-group imply that $\ker \pi$ is an irreducible $E$-group.  The fifth equality follows because $\ker \pi\isom G$ and $\ker\pi$ being an irreducible $E$-group and $(F/M\ra H) \isom (E,\pi)$ imply that $R/M$ is isomorphic to some $G_i$ as an $F$-group.

 Then we obtain the desired result by applying the same argument as at the end of Proposition~\ref{P:countNother}.
%Note that the data $(M,\phi)$ in the final equation above is exactly the same as the data of a surjection of $H$-extensions from $F\ra H$ to $E\ra H$.
%Now let $(E,\pi)$ be an $H$-extension with $\ker \pi$ (via conjugation) an irreducible $E$-group. Consider a surjection of $H$-extensions from $(\hat{F}_n)^{\bS}\ra H$ to $E\ra H$, in which the map  $(\hat{F}_n)^{\bS}\ra E$ has kernel $K$. Again, we let $N$ denote the kernel of $(\hat{F}_n)^{\bS}\ra H$.  Then $N/K$ is an irreducible $(\hat{F}_n)^{\bS}$-group, and so $K$ is a maximal proper $(\hat{F}_n)^{\bS}$-subgroup of $N$.  So the map $(\hat{F}_n)^{\bS}\ra E$ factors through $F$.  On the other hand, any surjection of $H$-extensions from $F\ra H$ to $E\ra H$ clearly extends to a surjection of $H$-extensions from $(\hat{F}_n)^{\bS}\ra H$ to $E\ra H$. 
%
%Thus, the above sum is equal to
%$$
%\sum_{\substack{\textrm{isom. classes of $H$-extensions $(E,\pi)$}\\ \textrm{$\ker \pi \simeq G$ }
%\\ \textrm{$\ker \pi$ irred. $E$-group} 
% }}
%\frac{|\Sur((\hat{F}_n)^{\bS}\ra H,\pi)|
%}{|\Aut_H(E,\pi)|}.
%$$
%Then by applying Proposition~\ref{P:countfromF}  we obtain the result.
\end{proof}

We now can determine the multiplicities of the non-abelian irreducible $F$-groups in $R$ by combining Theorem~\ref{T:countRsub} and Proposition~\ref{P:mulnonab}.

\begin{corollary}[Multiplicities of non-abelian $G$ in $R$]\label{C:finalmnon}
Let $H$, $F$ and $R$ be as above.
Let $G$ be an non-abelian finite group. 
Then 
$$m(S,n,H,G)=\sum_{\substack{\textrm{isom. classes of $H$-extensions $(E,\pi)$}\\ \textrm{$\ker \pi \simeq G$ }
\\ \textrm{$\ker \pi$ irred. $E$-group} \\
\textrm{$E$ level $S$} }} 
\frac{\sum_{\substack{D\in\mathcal{E}_H, D\leq E}} \nu(D,E) \left( \frac{|D|}{|H|} \right)^n}{|\Aut_H(E,\pi)|}.
$$
%In particular $m(S,n,H,G)$ does not depend on the choice of surjection  $(\hF_n)^{\bS} \ra H$.
\end{corollary}

Finally, before we prove Theorem~\ref{T:calc}, we need the following lemma, whose proof is straightforward.

\begin{lemma}\label{L:seq}
Suppose $x_1,x_2,\dots$ is a sequence of real numbers with limit $x$, and $y_1,\dots$ is a sequence of real numbers with limit $\infty$.  Let $a>1$ be a real number.  
Then $f(x)=\prod_{i=1}^\infty (1-xa^{-i})$ is continuous in $x$ and
$$
\lim_{n\ra \infty} \prod_{i=1}^{y_n} (1-x_na^{-i})=\prod_{i=1}^{\infty} (1-xa^{-i}).
$$
\end{lemma}

\begin{comment}
\begin{proof}
\melanie{Will eventually cut this proof and say straightforward.  Just leaving it in now so we can check it.}
We can assume without loss generality that $x$ is arbitrarily small, since the limit clearly holds for the product of the first $M$ factors
for any $M$, and thus we can reduce the problem to showing the lemma with $x_n$ and $x$ replaced by $x_n/a^M$ and $x/a^M$.

Since for any real number $a>1$, the sum $\sum_{i=1}^\infty |xa^{-i}|$ converges locally uniformly  (as a series of functions of $x$), we have that $f(x)=\prod_{i=1}^\infty (1-xa^{-i})$ is continuous.  So as long as $y_n$ is sufficiently large, we can make $\prod_{i=1}^{y_n} (1-xa^{-i})$ arbitrarily close to $\prod_{i=1}^{\infty} (1-xa^{-i})$ (since $\prod_{i=y_n+1}^{\infty} (1-xa^{-i})=f(xa^{-y_n})$).

We have
\begin{equation}\label{E:prod}
\frac{\prod_{i=1}^{y_n} (1-x_na^{-i})}{\prod_{i=1}^{y_n} (1-xa^{-i})}=\prod_{i=1}^{y_n}(1+ \frac{(x-x_n)a^{-i}}{1-xa^{-i}}),
\end{equation}
and for $|xa^{-1}|\leq 1/2$ and $|(x-x_n)|a^{-1}\leq 1/8$ , we have that
$$
e^{-4|(x-x_n)|/(a-1)} \leq \prod_{i=1}^{y_n} e^{-4|(x-x_n)|a^{-i}}  \leq \prod_{i=1}^{y_n}(1+ \frac{(x-x_n)a^{-i}}{1-xa^{-i}}) \leq \prod_{i=1}^{y_n} e^{2|(x-x_n)|a^{-i}}\leq e^{2|(x-x_n)|/(a-1)}
$$
So by taking $n$ sufficiently large, we can make the product in Equation~\eqref{E:prod} arbitrary close to $1$.  From this the lemma follows.
\end{proof}
\end{comment}

\begin{proof}[Proof of Theorem~\ref{T:calc}]
Equation~\eqref{E:setup} and Theorem~\ref{T:probfrommult} establish Equation~\eqref{E:fixedn} of Theorem~\ref{T:calc}. 
%Let's first establish Equation~\eqref{E:fixedn} of Theorem~\ref{T:calc}. The number of normal subgroups $N$ of 
%$(\hat{F}_n)^{\bS}$ such that $(\hat{F}_n)^{\bS}/N\isom H$ is $|\Sur(\hat{F}_n,H)|/|\Aut(H)|$. For each of these subgroups, the %probability that the independent, uniform random elements  $r_1,\dots,r_{n+u}$ of $(\hat{F}_n)^{\bS}$ generate $N$ as an $F$-normal %subgroup is $|H|^{-(n+u)}$ (for the probability that all $r_i\in N$) times the probability that $n+u$ independent, uniform random %elements  of $N$ generate $N$ as a $F$-normal subgroup.  This latter probability is given by Theorem~\ref{T:probfrommult}, combined with %the multiplicities given by Corollaries~\ref{C:finalmab} and \ref{C:finalmnon}. 
Recall that for $G$ a finite abelian irreducible $H$-group, we have defined
$$
\lambda(S,H,G):=(h_H(G)-1) 
\sum_{\substack{\textrm{isom. classes of $H$-extensions $(E,\pi)$}\\ \textrm{such that $\ker \pi \simeq G$ as  $H$-groups,} \\
\textrm{and $E$ is level $S$}  }}
\frac{1}{|\Aut_H(E,\pi)|}
$$
and Corollary~\ref{C:finalmab} gives that
$$
\frac{h_H(G)^{m(S,n,H,G)}-1}{h_H(G)-1} =\sum_{\substack{\textrm{isom. classes of $H$-extensions $(E,\pi)$}\\ \textrm{$\ker \pi \simeq G$ as an $H$-group} \\
\textrm{$E$ is level $S$} }}
\frac{\sum_{\substack{D\in\mathcal{E}_H, D\leq E}} \nu(D,E) \left( \frac{|D|}{|H|} \right)^n}{|\Aut_H(E,\pi)|}.
$$
So we have
\begin{align*}
\lim_{n\ra \infty} \frac{h_H(G)^{m(S,n,H,G)}}{|G|^n} &=\lim_{n\ra \infty} \frac{h_H(G)^{m(S,n,H,G)}-1}{|G|^n}\\
&=
 \lim_{n\ra \infty} \frac{h_H(G)-1}{|G|^n} \sum_{\substack{\textrm{isom. classes of $H$-extensions $(E,\pi)$}\\ \textrm{$\ker \pi \simeq G$ as an $H$-group}\\
\textrm{$E$ is level $S$} }}
\frac{\sum_{\substack{D\in\mathcal{E}_H, D\leq E}} \nu(D,E) \left( \frac{|D|}{|H|} \right)^n}{|\Aut_H(E,\pi)|}. 
%&=
%\lambda(S,H,G)
\end{align*}
Note for any $D\ne E$ in the above sum, we have $\lim_{n\ra\infty} |D|^{n}|H|^{-n}|G|^{-n}=0$.
Thus we conclude (and similarly in the non-abelian case using Corollary \ref{C:finalmnon}) that $\lambda(S,H,G)$ is related to $m(S,n,H,G)$ as follows:
\begin{align}\label{E:relmandlam}
\lim_{n\ra \infty} \frac{h_H(G)^{m(S,n,H,G)}}{|G|^n}% &=\lim_{n\ra \infty} (h(G)-1) 
%\sum_{\substack{\textrm{isom. classes of $H$-extensions $(E,\pi)$}\\ \textrm{$\ker \pi$ isom. $G$ as an $H$-%group} \\
%\textrm{$E$ quasi-level $S$}  }}
%\frac{\sum_{\substack{D \textrm{ subgroup of } E\\\pi(S)=H}} \nu(S,E) \left( \frac{|D|}{|E|} \right)^n}{\#
%\Aut_H(E,\pi)}\\
&= \lambda(S,H,G) \quad\quad &\textrm{for $G\in \mathcal{A}_H$}\\
\lim_{n\ra\infty} 
\frac{m(S,n,H,G)}{|G|^n} &= \lambda(S,H,G) \quad\quad &\textrm{for $G\in \mathcal{N}$}. \notag
\end{align}
\begin{remark}\label{R:lamint}
Note that since by Remark~\ref{R:Gpow}, for $G\in\mathcal{A}_H$, we have that $|G|$ is a power of $h_H(G)$, it follows that $\lambda(S,H,G)$ is an integral power of 
$h_H(G)$, and that the limit above stabilizes for sufficiently large $n$.
%This is for finite $S$, but for infinite $S$, we will also have that $m$ is a positive integer, and that this limit holds, and the
%same conclusion holds.
\end{remark}
We next establish the final statement of Theorem~\ref{T:calc}.  Since
any irreducible $F$-group factor of $R$ is in $\bar{S}$, Corollary~\ref{C:nonewsimple} shows that it is a power of a simple group in the closure of $S$ under taking subgroups and quotients.  
 Lemma~\ref{L:HfactR} shows that the power is bounded by $|H|,$ showing the final statement of Theorem~\ref{T:calc}. 

To establish Equation~\eqref{E:lim}, it will suffice to take the limit of a factor in Equation~\eqref{E:fixedn}
corresponding to a single $G$ (since there are only finitely many $G$ with non-trivial factors, independent of $n$, by Lemma~\ref{L:R-cfp} and Corollary~\ref{C:CfpFin}). 
The factor in Equation~\eqref{E:fixedn} for a $G\in\mathcal{A}_H$ is
$$
\prod_{k=0}^{m(S,n,H,G)-1} (1-\frac{h_H(G)^k}{ |G|^{{n+u}}})=\prod_{i=1}^{m(S,n,H,G)} (1-\frac{h_H(G)^{m(S,n,H,G)} h_H(G)^{-i}}{ |G|^{{n+u}}}).
$$
If there are no extensions $E$ in the sum in Corollary~\ref{C:finalmab}, then $m(S,n,H,G)$ and $\lambda(S,H,G)$ are $0$. Otherwise $\lambda(S,H,G)>0$, and thus it follows from Equation~\eqref{E:relmandlam} that $m(S,n,H,G)\ra\infty$ as $n\ra\infty$.
So using Lemma~\ref{L:seq} and Equation~\eqref{E:relmandlam}, we obtain the limit in Equation~\eqref{E:lim} for a single factor $G\in\mathcal{A}_H$.  In a similar but simpler fashion, from Equation~\eqref{E:relmandlam}, we obtain the limit in Equation~\eqref{E:lim} for a single factor $G\in\mathcal{N}$.  This completes the proof of Theorem~\ref{T:calc}.
\end{proof}

\section{Countable additivity of $\mu_u$}\label{S:cntadd}

The goal of this section is to prove Theorem~\ref{T:countadd} which states that $\mu_u$ defined in Equation~\eqref{E:defmu} is countably additive on the algebra $\mathcal{A}$.
It then follows from Carath\'{e}odory's extension theorem that $\mu_u$ can be uniquely extended to a measure on the Borel sets of $\mathcal{P}$.
% This will directly imply that $\mu_u$ is a pre-measure on the algebra $\mathcal{A}$. 
The heart the proof of Theorem~\ref{T:countadd} is Theorem~\ref{T:totall}. We will first prove Theorem~\ref{T:totall} in Section~\ref{SS:proof-totall}, and then prove Theorem~\ref{T:countadd} in Section~\ref{SS:countadd}.

\begin{theorem}\label{T:countadd}
Let $u$ be an integer.  Then $\mu_u$ is countably additive on the algebra $\mathcal{A}$ generated by the $U_{S,H}$ for $S$ a finite set of finite groups and $H$ a finite group.
\end{theorem}

\begin{theorem}\label{T:totall}
Let $\ell$ be a positive integer. Recall that $S_{\ell}$ is defined to be the set consisting of all groups of order less than or equal to $\ell$. For a non-negative integer $u$, we have
$$\sum_{H \text{ is finite and level }\ell} \mu_u (U_{S_{\ell}, H}) = 1. $$
\end{theorem}

\subsection{Proof of Theorem \ref{T:totall}}\label{SS:proof-totall}
Assume $\ell$ is a positive integer, $H$ is a finite level $\ell$ group and let $\tilde{H}=H^{\bS_{\ell-1}}$ .
%and We denote the pro-$\bS_{\ell-1}$ completion of $H$ by $\widetilde{H}$ and $\chi: H \to \widetilde{H}$ be the canonical surjection. \melanie{why not take the canonical one?}
In Lemmas~\ref{L:ab-con} and \ref{L:nab-con}, we will first give upper bounds for the number of irreducible factors $G$ with non-zero $m(S_{\ell}, n, H, G)$ for some $n$ that are isomorphic to a given underlying group $M$.

\begin{lemma}\label{L:ab-con}
 Suppose $M$ is a direct product of isomorphic abelian simple groups. Then 
 $$\#\left\{ G\in \mathcal{A}_H \,\Bigg|\, 
 \begin{aligned}
 & G\simeq M \text{ and }\\
 & m(S_\ell, n, H, G)\neq 0\text{ for some }n
 \end{aligned}
 \right\} \leq \sum_{(M,A)\in \CF(\bS_{\ell})} |\Sur (\widetilde{H}, A) |,$$
 where the notation on the right-hand side above means that the sum is taken over all chief factor pairs in $\CF(\bS_\ell)$ whose first components are isomorphic to $M$ as groups.
\end{lemma}

\begin{proof}
We give an injection  
 $$\left\{ G\in \mathcal{A}_H, \phi \,\Bigg|\, 
 \begin{aligned}
 & \phi: G\simeq M \text{ and }\\
 & m(S_\ell, n, H, G)\neq 0\text{ for some }n
 \end{aligned}
 \right\} \ra \{(M,A)\in \CF(\bS_{\ell}),\pi  \,|\, \pi\in\Sur (H, A)\}.$$
 Consider $G\in \mathcal{A}_H$ and $\phi:G\simeq M$ such that $m(S_\ell,n,H,G)\neq 0$ for some $n$. Assume 
$$1 \ra R \ra F \ra H \ra 1$$
is the fundamental short exact sequence associated to $S_\ell,$ $n,$ and $H$. Then $G$ appears as a factor in $R$ and $(G, \rho_{G}(F))\in \CF(\bS_{\ell})$. %. by Lemma~\ref{L:R-cfp}.
  Using $\phi:G\simeq M$, we have that the quotient $\rho_{G}(F)$ of $H$ acts on $M$, and so $(M, \rho_{G}(F))\in \CF(\bS_{\ell}).$
We let $\pi$ be the quotient map from $H$ to $\rho_{G}(F)$.  Given $(M,A)\in \CF(\bS_{\ell})$ and $\pi\in\Sur (H, A)$, we can use $\pi$ to give $M$ the structure of an irreducible $H$-group, and let $\phi$ be the identity.  This recovers $G$ and $\phi$, though possibly without $m(S_\ell, n, H, G)\neq 0$.
By Corollary \ref{C:CfpFin}, if $(M,A)\in \CF(\bS_\ell)$, then $A/\Inn(M)\simeq A$ is a group of level $\ell-1$. Then by the definition of pro-$\bS$ completion, $\Sur(H, A)$ is one-to-one corresponding to $\Sur(\widetilde{H}, A)$ and we finish the proof.
\end{proof}

Similarly, for non-abelian irreducible factors, we have the following lemma.

\begin{lemma}\label{L:nab-con}
Suppose $M$ is a direct product of isomorphic non-abelian simple groups. Then 
\begin{eqnarray*}
 && \# \left\{ \text{isom. classes of $H$-extensions } (E,\pi)\, \Bigg|\,
  \begin{aligned}
  & \ker\pi \simeq M \text{ is irred. $E$-group}\\
  & E \text{ is level }\ell
  \end{aligned}
\right\} \\
&\leq& \sum_{(M,A)\in \CF(\bS_\ell)} |\Sur(\widetilde{H},  A/\Inn(M))|.
\end{eqnarray*}
\end{lemma}

\begin{proof}
We give an injection
\begin{eqnarray*}
  && \left\{ \text{isom. classes of $H$-extensions } (E,\pi)\, \Bigg|\,
  \begin{aligned}
  & \ker\pi \simeq M \text{ is irred. $E$-group}\\
  & E \text{ is level }\ell
  \end{aligned}
\right\} \\
&\to& \{(M,A)\in \CF(\bS_{\ell}), \phi \mid \phi \in \Sur(H,A/\Inn(M))\}.
\end{eqnarray*}
Consider an isomorphism class of $H$-extension $(E,\pi)$ such that $\ker\pi\simeq M$ is an irreducible $E$-group and $E$ is level $\ell$. Then $(\ker \pi, \rho_{\ker\pi}(E)) \in \CF(\bS_{\ell})$, and $\rho_{\ker\pi}$ induces a surjection $\phi: H\to \rho_{\ker \pi}(E)/\Inn(M)$ since $\rho_{\ker \pi}$ is an isomorphism when restricted on $\ker\pi$ that maps $\ker\pi$ to $\Inn(M)$. Suppose $(M,A)\in \CF(\bS_{\ell})$ and $\phi\in \Sur(H, A/\Inn(M))$. If two $H$-extensions $(E_1, \pi_1)$ and $(E_2, \pi_2)$ both map to $(M,A)$ and $\phi$, then from the diagram below we see that $E_1$ and $E_2$ are both the fiber product of $\phi$ and $A\to A/\Inn(M)$, so $(E_1,\pi_1)$ and $(E_2, \pi_2)$ are isomorphic as $H$-extensions. Therefore, the map defined at the begin of the proof is an injection. Then the lemma follows as $\Sur(H,A)$ is one-to-one corresponding to $\Sur(\widetilde{H},A)$.

\begin{center}
\begin{tikzcd}
 E_2
 \arrow[bend left]{drr}{\pi_2}
 \arrow[swap, bend right]{ddr}{\rho_{\ker\pi_2}}
 \arrow[dashed]{dr}{\iota}& & \\
 & E_1 \arrow{r}{\pi_1}
 \arrow{d}{\rho_{\ker \pi_1}}& 
 H \arrow{d}{\phi} \\
 & A \arrow{r}{} & 
 A/\Inn(M) 
\end{tikzcd}
\end{center}
\end{proof}

Let $P_{u,n}(U_{S_\ell, H})$ denote the product in Equation~\eqref{E:fixedn}, i.e.
$$P_{u,n}(U_{S_\ell,H}) = \prod_{G\in \mathcal{A}_H} \prod_{k=0}^{m(S_\ell, n, H, G)-1} (1-\frac{h_H(G)^k}{|G|^{n+u}}) \prod_{G\in \mathcal{N}}(1-|G|^{-n-u})^{m(S_\ell, n, H,G)}.$$

\begin{lemma}\label{L:P-Const}
Suppose $\ell>1$ ,$n\geq 1$ and $u> -n$ are integers and $\widetilde{H}$ is a finite level $\ell-1$ group. Then there exists a non-zero constant $c(u,\ell,\widetilde{H})$ depending on $u, \ell$ and $\widetilde{H}$ such that, for every finite level $\ell$ group $H$ with $H^{\bS_{\ell-1}}=\widetilde{H}$,
 either $P_{u,n}(U_{S_\ell,H})\geq c(u,\ell,\widetilde{H})$ or $P_{u,n}(U_{S_\ell,H})=0$ .
\end{lemma}

\begin{proof}
For each $G\in \mathcal{A}_H$, $G$ is a direct product of isomorphic abelian simple groups, i.e. $G$ is a direct product of $\Z/p\Z$ for some prime $p$. By Remark \ref{R:Gpow}, $h_H(G)$ is a power of $p$ and $|G|$ is a power of $h_H(G)$. Note that both of the trivial map (every element maps to 1) and the identity map of $G$ respect the $H$-action, therefore $h_H(G)>1$. So if the product $\prod_{k=0}^{m(S_\ell, n, H, G)-1}(1-\frac{h_H(G)^k}{|G|^{n+u}})$ is nonzero, then it is a product of $1-p^{-i}$ for distinct positive integers $i$, which is greater than $\prod_{k=1}^{\infty}(1-p^{-k}) \geq \prod_{k=1}^{\infty}(1-2^{-k})$. If $P_{u,n}(U_{S_\ell,H})\neq 0$, then
\begin{eqnarray*}
\prod_{G \in \mathcal{A}_H} \prod_{k=0}^{m(S_\ell, n, H, G)-1}(1-\frac{h_H(G)^k}{|G|^{n+u}})
  &=& \prod_{\substack{G\in \mathcal{A}_H \text{ and }\\ m(S_\ell, n, H, G)\neq 0}} \prod_{k=0}^{m(S_\ell, n, H, G)-1}(1-\frac{h_H(G)^k}{|G|^{n+u}})\\
  &\geq& \prod_{\substack{G\in \mathcal{A}_H \text{ and }\\ m(S_\ell, n, H, G)\neq 0 \text{ for some }n}} \prod_{k=1}^{\infty} (1-2^{-k})\\
  &=& \left[\prod_{k=1}^{\infty}(1-2^{-k})\right]^{\#\{G\in\mathcal{A}_H \mid\, m(S_\ell, n, H, G)\neq 0 \text{ for some }n\}}\\
&\geq& \left[\prod_{k=1}^{\infty}(1-2^{-k})\right]^{\sum\limits_{\substack{(M,A)\in \CF(\bS_\ell)\\M \text{ is abelian}}} |\Sur(\widetilde{H}, A)|},
\end{eqnarray*}
where the last inequality follows from Lemma~\ref{L:ab-con}.
Therefore, if $P_{u,n}(U_{S_\ell,H})$ is non-zero, then its abelian part has a lower bound depending only on $\ell$ and $\widetilde{H}$. Similarly, for the non-abelian part, we consider
\begin{eqnarray*}
  \prod_{G \in \mathcal{N}} (1-|G|^{-n-u})^{m(S_\ell,n,H,G)}
  &\geq& \prod_{G\in \mathcal{N}} \left[(1-\frac{1}{2})^2\right]^{\frac{m(S_\ell,n,H,G)}{|G|^{n+u}}}\\
  &=& \left[(1-\frac{1}{2})^2\right]^{\sum\limits_{G\in \mathcal{N}}\frac{m(S_\ell,n,H,G)}{|G|^{n+u}}},
\end{eqnarray*}
where the first inequality follows because $(1-\frac{1}{n})^n$ is an increasing sequence. Then we have
\begin{eqnarray*}
  &&\sum_{G\in \mathcal{N}} \frac{m(S_\ell,n,H,G)}{|G|^{n+u}}\\
  &=& \sum_{G\in \mathcal{N}} |G|^{-u}\left(\sum_{\substack{\text{isom. classes of }H\text{-extensions }(E,\pi) \\ \ker\pi \simeq G\\ \ker \pi  \text{ irred. $E$-group}\\ E \text{ is level }\ell}} \frac{|G|^{-n}\sum_{D\in \mathcal{E}_H, D\leq E} \nu(D,E)\frac{|D|}{|H|}^n}{|\Aut_H(E,\pi)|} \right) \\
  &\leq& \sum_{G \in \mathcal{N}} |G|^{-u} \# \left\{ \text{isom. classes of $H$-extensions } (E,\pi) \,\Bigg|\,
  \begin{aligned}
  & \ker\pi \simeq G \text{ is irred. $E$-group}\\
  & E \text{ is level }\ell
  \end{aligned}
\right\} \\
  &\leq& \sum_{G \in \mathcal{N}} |G|^{-u} \left(\sum_{(G,A)\in \CF(\bS_\ell)} |\Sur(\widetilde{H}, A/\Inn(G))| \right)\\
  &=& \sum_{\substack{(G,A)\in \CF(\bS_\ell)\\ G \text{ non-abelian}}} |G|^{-u} |\Sur(\widetilde{H}, A/\Inn(G))|.
\end{eqnarray*}
The first equality above is Corollary~\ref{C:finalmnon}. The first inequality follows from the fact that $|\Sur(\rho, \pi)|$ in Proposition~\ref{P:countfromF} is less than or equal to $|G|^n$. The second inequality follows by Lemma~\ref{L:nab-con}.
It shows that the non-abelian part also has a lower bound depending on $u, \ell$ and $\widetilde{H}$. By Corollary~\ref{C:CfpFin}, $\CF(\bS_\ell)$ is a finite set, so these lower bounds for abelian part and non-abelian parts are both non-zero. Then we proved the theorem.
\end{proof}

Now, we establish the inductive step that is crucial in the proof of Theorem~\ref{T:totall}.

\begin{lemma} \label{L:l-1tol}
Let $\ell>1$, $n\geq 1$, $u> -n$ be integers, and $\widetilde{H}$ be a finite level $\ell-1$ group. Then
$$\lim_{n\to \infty} \sum_{\substack{H \text{ is finite level }\ell\\ \text{s.t. }\widetilde{H}=H^{\bS_{\ell-1}}}} \mu_{u,n}(U_{S_\ell,H})=\sum_{\substack{H \text{ is finite level }\ell\\ \text{s.t. }\widetilde{H}=H^{\bS_{\ell-1}}}}\mu_u (U_{S_\ell,H}).$$
\end{lemma}

\begin{proof}

Assume $H$ is finite and level $\ell$ such that $\widetilde{H}=H^{\bS_{\ell-1}}$. 
Let $i(H)$ be the smallest integer such that $|\Sur(\hat{F}_n, H)|/|H|^n\geq \frac{1}{2}$ for all $n>i(H)$ (note that $i(H)$ is finite since $\lim_{n\to \infty}|\Sur(\hat{F}_n, H)|/|H|^n =1$).
Then either $\mu_{u,n}(U_{S_\ell,H})=0$ or
\begin{eqnarray*}
  \mu_{u,n}(U_{S_\ell,H}) &=& \frac{|\Sur(\hat{F}_n,H)|}{|\Aut(H)||H|^{n+u}} P_{u,n}(U_{S_\ell,H})\\
  &\geq& \frac{1}{2}c(u, \ell, \widetilde{H}) \frac{1}{|\Aut(H)||H|^u}
\end{eqnarray*}
for $n>i(H)$, by Lemma~\ref{L:P-Const}. 
We call $H$ \emph{achievable} if it is finite level $\ell$ and there exists $n$ such that $\mu_{u,n}(U_{S_\ell, H})\neq 0$ (we will give an equivalent definition in Section~\ref{S:prob0}).
The function $\mu_{u,n}(U_{S_\ell,H})$ of $H$ is dominated by the function of $H$ that is $\frac{1}{|\Aut(H)||H|^u}$ when $H$ is achievable and $0$ otherwise.
We will next show that the sum of this dominating function converges, in order to use Lebesgue's Dominated Convergence Theorem.
We have
\begin{eqnarray*}
\sum_{\substack{H \text{ is achievable} \\ \text{ s.t. } \widetilde{H}\simeq H^{\bS_{\ell-1}}}} \frac{1}{|\Aut(H)||H|^u}
  &=& \lim_{n\to\infty}\sum_{\substack{H \text{ is achievable} \\ \text{ s.t. } \widetilde{H}\simeq H^{\bS_{\ell-1}}\\ \text{and } i(H)<n}} \frac{1}{|\Aut(H)||H|^u}\\ 
  &\leq& \lim_{n\to \infty}\frac{2}{c(u,\ell, \widetilde{H})} \sum_{\substack{H \text{ is achievable} \\ \text{ s.t. } \widetilde{H}\simeq H^{\bS_{\ell-1}}\\ \text{and } i(H)<n}} \mu_{u,n}(U_{S_\ell,H})\\
  &\leq& \frac{2}{c(u,\ell, \widetilde{H})},
\end{eqnarray*}
where the first equality expresses the implicit infinite sum as an explicit limit.
Thus by Lebesgue's Dominated Convergence Theorem, 
\begin{eqnarray*}
 \lim_{n\to \infty} \sum_{\substack{H \text{ is finite level }\ell \\ \text{ s.t. } \widetilde{H}\simeq H^{\bS_{\ell-1}}}} \mu_{u,n}(U_{S_\ell,H})
 &=&  \sum_{\substack{H \text{ is finite level }\ell \\ \text{ s.t. } \widetilde{H}\simeq H^{\bS_{\ell-1}}}} \lim_{n\to \infty} \mu_{u,n}(U_{S_\ell,H}),
\end{eqnarray*}
which completes the lemma.
\end{proof}

\begin{proof}[Proof of Theorem \ref{T:totall}]
We proceed by induction on $\ell$. When $\ell=1$, note that the trivial group is the only group that is finite level 1 and it's obvious that $\mu_u(U_{S_1,1})=1$. Assume the theorem is true for $\ell-1$, i.e. $$\sum\limits_{\widetilde{H} \text{ is finite level }\ell-1} \mu_u(U_{S_{\ell-1},\widetilde{H}})=1.$$ 
 We see that for any finite level $\ell-1$ group $\widetilde{H}$
 \begin{eqnarray*}
   \mu_u(U_{S_{\ell-1}, \widetilde{H}}) &=& \lim_{n\to \infty} \mu_{u,n}(U_{S_{\ell-1},\widetilde{H}}) \\
   &=& \lim_{n\to \infty} \sum_{\substack{H \text{ is finite level }\ell \\ \text{s.t. }\widetilde{H}=H^{\bS_{\ell-1}}}} \mu_{u,n}(U_{S_\ell,H})\\
   &=& \sum_{\substack{H \text{ is finite level }\ell \\ \text{s.t. }\widetilde{H}=H^{\bS_{\ell-1}}}} \mu_u(U_{S_\ell,H}),
 \end{eqnarray*}
 where the second equality above follows from the definition of $\mu_{u,n}$ on basic open sets and the last step follows from Lemma~\ref{L:l-1tol}.
 Therefore, we finish the proof by
 \begin{eqnarray*}
   \sum_{H \text{ is finite level }\ell} \mu_u(U_{S_\ell,H}) &=& \sum_{\widetilde{H} \text{ is finite level }\ell-1} \sum_{\substack{H \text{ is finite level }\ell \\ \text{s.t. }\widetilde{H}=H^{\bS_{\ell-1}}}} \mu_u(U_{S_\ell,H})\\
   &=& \sum_{\widetilde{H} \text{ is finite level }\ell-1} \mu_u(U_{S_{\ell-1},\widetilde{H}})\\
   &=& 1.
 \end{eqnarray*}
\end{proof}

\subsection{Proof of Theorem~\ref{T:countadd}}\label{SS:countadd}

We will use the following corollary of Theorem~\ref{T:totall}.
\begin{corollary}\label{C:addB}
Let $\ell$ be a positive integer, and $B=\cup_{j=1}^{\infty} U_{S_\ell,H_j}$ for some finite groups $H_j$ such that
$U_{S_\ell,H_j}\ne U_{S_\ell,H_{j'}}$ for $j\ne j'$ (note that $U_{S_\ell, H_{j}}\neq U_{S_{\ell},H_{j'}}$ implies $U_{S_\ell, H_{j}}\cap U_{S_{\ell},H_{j'}}=\emptyset$).  Suppose that $B\in \mathcal{A}$, the algebra of sets generated by the basic open sets $U_{S,G}$ for a finite set $S$ of finite groups and a finite level $S$ group $G$.
Let $u$ be an integer.
Then $\mu_u(B)=\sum_{j=1}^\infty \mu_u(U_{S_\ell,H_j})$.
\end{corollary}
\begin{proof}
Since $\mu_u$ is defined as a limit of measures $\mu_{u,n}$, it is immediate that $\mu_u$ is finitely additive because finite sums can be exchanged with the limit.
Let $G_j$ be the level $\ell$ finite groups not among the $H_j$.  Then for every positive integer $M$, we have
$$
\sum_{j=1}^M \mu_u(U_{S_\ell,H_j}) \leq \mu_u(B) \leq 1-\sum_{j=1}^M \mu_u(U_{S_{\ell},G_j}).
$$
Taking limits as $M\ra\infty$ gives
$$
\sum_{j=1}^\infty \mu_u(U_{S_\ell,H_j}) \leq \mu_u(B) \leq 1-\sum_{j=1}^\infty \mu_u(U_{S_\ell,G_j})=\sum_{j=1}^\infty \mu_u(U_{S_\ell,H_j}),
$$
where the last equality is by Theorem~\ref{T:totall}.
\end{proof}

\begin{proof}[Proof of Theorem~\ref{T:countadd}]
 If we have disjoint sets $A_n \in\mathcal{A}$ with $A=\cup_{n\geq 1 } A_n \in \mathcal{A}$, by taking $B_n=A\setminus \cup_{j=1}^n A_j$, it suffices to show that for $B_1 \supset B_2 \supset \dots$ (with $B_n\in \mathcal{A}$) with $\cap_{n\geq 1} B_n=\emptyset$ we have $\lim_{n\ra\infty} \mu_u(B_n)=0$.

We can assume, without loss of generality, that 
 for each $\ell\geq 1$, we have $B_\ell=\cup_{j} U_{S_\ell,G_{\ell,j}}$ (i.e. $B_\ell$ is defined at level $\ell$).
(Note that when all groups in $S$ have order at most $m$ that $U_{S,H}$ is a union of sets of the form $U_{S_m,G}$ for varying $G$.  We can always insert redundant $B_i$'s if the level required to define the $B_\ell$ increase quickly.) 
We will show by contradiction that $\lim_{\ell \ra \infty} \mu_u(B_\ell)=0$.

Suppose, instead that there is an $\epsilon>0$ such that for all $\ell$, we have $\mu_u(B_\ell)\geq \epsilon$.
It follows from Corollary~\ref{C:addB}  that for each $\ell$ we have a subset $K_\ell\sub B_\ell$ such that $\mu_u(B_\ell\setminus K_\ell)<\epsilon/2^{\ell+1}$ and $K_\ell$ is a \emph{finite} union of $U_{S_\ell,G_{\ell,j}}$. 

Next, let $C_\ell=\cap_{j=1}^\ell K_j$.    Then 
$\mu_u(B_\ell\setminus C_\ell)<\epsilon/2$, since
\begin{align*}
\mu_u(B_\ell\setminus C_\ell)=& \mu_u(B_\ell\setminus K_\ell)+\mu_u(K_\ell\setminus K_\ell \cap K_{\ell-1})+\cdots 
+\mu_u(K_\ell \cap \cdots \cap K_2 \setminus K_\ell \cap \cdots \cap K_1) \\
< &\epsilon/2^{\ell+1} + \mu_u(B_{\ell-1} \setminus K_{\ell-1}) + \cdots   +\mu_u(B_1\setminus K_{1}) \\
< &\epsilon/2^{\ell+1} + \epsilon/2^{\ell} + \cdots   +\epsilon/2^{2}.
\end{align*}
So $\mu_u(C_\ell)\geq \epsilon/2$ for each $\ell$ and in particular it is non-empty.  Note $C_{\ell+1}\sub C_\ell$ for all $\ell$.  Pick $x_\ell\in C_\ell$ for all $\ell$.

Note $C_\ell$ is defined at level $\ell$ and a finite union of the basic open sets $U_{S_\ell,G_{\ell,j}}$.
Pick an $H_1$ so that infinitely many of the $x_\ell$ are in $U_{S_1,H_1}$ (this is possible since all $x_\ell$ are in $C_1$ and there are only finitely many $U_{S_1,H}$ that make up $C_1$), and then disregard the $x_\ell$ that are not in $U_{S_1,H_1}$.  In particular note $U_{S_1,H_1}\subset C_1$. Then pick $H_2$ so that infinitely many of the remaining $x_\ell$ are in $U_{S_2,H_2}$, and disregard the $x_\ell$ that are not. Since all of the remaining $x_\ell$ are in $U_{S_1,H_1}$, we have $U_{S_2,H_2}\sub U_{S_1,H_1}$ and hence $H_1$ is a quotient of $H_2$.  Also note $U_{S_2,H_2}\sub C_2$.  We continue this process and then consider the profinite group $H$ that is the inverse limit of the $H_i$'s.  Since $H\in U_{S_\ell,H_\ell} \sub C_\ell \sub B_\ell$ for all $\ell$, we have a point $H\in \cap_{\ell\geq 1} B_\ell $ which is a contradiction.  
 
\end{proof}

\section{Proof of Theorem~\ref{T:Main}}\label{S:mainproof}
The last section established the existence of the probability measure $\mu_u$ on Borel sets of $\mathcal{P}$.
Now we are able to give the proof of Theorem~\ref{T:Main}, the weak convergence of the $\mu_{u,n}$ to $\mu_u$.

\begin{proof}[Proof of Theorem~\ref{T:Main}]

%Theorem~\ref{T:countadd} shows that $\mu_u$ is a pre-measure on the algebra $\mathcal{A}$ of sets generated by basic open sets $U_{S,H}%
%$. Then by the Carath\'eodory's extension theorem, $\mu_u$ can be extended to a measure (which will also be denoted by $\mu_u$) on the $
%\sigma$-algebra generated by $\mathcal{A}$, i.e. the $\sigma$-algebra of Borel sets of $\mathcal{P}$.

Note that the weak convergence $\mu_{u,n} \Rightarrow \mu_u$ is equivalent to that $$\liminf_{n\to\infty} \mu_{u,n} (U) \geq \mu_u(U)$$ for all open sets $U$. In the topological space $\mathcal{P}$, every open set is a countable disjoint union of basic open sets, since two open basic open sets having nontrivial intersection implies that one basic open set contains the other. Assume $U=\cup_{i\geq 1} U_i$ is an open set, where $U_i$ are disjoint basic open sets. By Fatou's lemma, we have
 $$\mu_u(U) = \sum_{i\geq 1} \mu_u(U_i)
 = \sum_{i\geq 1} \lim_{n\to \infty} \mu_{u,n} (U_i)
 \leq \liminf_{n\to \infty }\sum_{i\geq 1} \mu_{u,n}(U_i) 
 = \liminf_{n\to \infty} \mu_{u,n}(U).$$
\end{proof}

\section{For arbitrary set $S$} \label{S:arbitraryS}

In this section, we let $S$ be an arbitrary (not necessarily finite) set of finite groups and consider the value of $\mu_u$ on the specific type of Borel sets 
$$V_{S,H}:=\{X \in\mathcal{P} \mid X^{\bS}\simeq H\}$$
for a finite level $S$ group $H$. We will first prove an analogue of Theorem~\ref{T:calc} for an arbitrary set $S$ (see Theorem~\ref{T:calcinf}), the proof of which shows that Equation~\eqref{E:lim} gives the value $\mu_{u}(V_{S,H})$.

Note that $V_{S,H}$ is not a basic open set, but is the intersection of a sequence of basic open sets. Since we will approximate $S$ by increasing finite subsets, we need the following lemma.

\begin{lemma}\label{L:mincinS}
Consider two sets $T\sub T'$ of finite groups.  For any positive integer $n$, finite group $H$ of level $T$, and $G\in \mathcal{A}_H\cup \mathcal{N}$, we have
$m(T,n,H,G)\leq m(T',n,H,G).$  Also if $T_1\sub T_2\sub \cdots$ are finite sets of finite groups, then $m(T_m,n,H,G)$ eventually stabilizes as $m\ra\infty$.
\end{lemma}
\begin{proof}
Consider the case when $G$ is abelian.  
Let $\rho: \hF^{\bar{T}_m} \ra H$ be a surjection.
Corollary~\ref{C:finalmab} and Proposition~\ref{P:countfromF} give
$$
\frac{h_H(G)^{m(T_m,n,H,G)}-1}{h_H(G)-1} =\sum_{\substack{\textrm{isom. classes of $H$-extensions $(E,\pi)$}\\ \textrm{$\ker \pi \simeq G$ as an $H$-group} \\
\textrm{$E$ is level $T_m$} }}
\frac{|\Sur(\rho,\pi)|}{|\Aut_H(E,\pi)|}.
$$
The right-hand side is clearly non-decreasing in $m$.  There are only finitely many isomorphism classes of $H$-extensions whose kernel is isomorphic to $G$, which proves the stabilization.  The case of non-abelian $G$ is similar.
\end{proof}

\begin{definition}
Let $S$ be a set of finite groups, $n$ a positive integer, and $H$ a finite level $S$ group. Let $T_1\subset T_2 \subset \cdots$ be finite sets of finite groups such that $\cup_{m\geq 1}T_m=S$. For any $G\in \mathcal{A}_H \cup \mathcal{N}$, we define 
$m(S,n,H,G)= \lim_{m\to \infty} m(T_m, n, H, G).$
\end{definition}

\begin{remark}
It's clear that $m(S,n,H,G)$ does not depend on the choice of the increasing sequence $T_i$, and $m(S,n,H,G)$ is always a non-negative integer.
\end{remark}

It is actually easier to determine $\mu_u(V_{S,H})$, as we will in the next lemma, than to 
find $\lim_{n\ra\infty} \mu_{u,n}(V_{S,H}),$ which we will do in Theorem~\ref{T:calcinf}.

\begin{lemma}\label{L:wrongorder}
Let $S$ be a set of finite groups.
Let $T_1\sub T_2\sub \cdots$ be finite sets of finite groups such that $\cup_{m\geq 1} T_m=S$.
Let $H$ be a finite group of level $S$. Let $u$ be an integer.
Then
\begin{eqnarray*}
 \mu_u(V_{S,H})
&=&\lim_{m\ra\infty} \lim_{n\ra\infty} \Prob((X_{u,n})^{\bT_m}\isom H)\\
&=&\frac{1}{|\Aut(H)||H|^{u}}\prod_{\substack{G \in \mathcal{A}_H}} \prod_{i=1}^{\infty} (1-\lambda(S,H,G)
\frac{h_H(G)^{-i}}{ |G|^{{u}}})
  \prod_{\substack{ G\in \mathcal{N}}}  e^{-|G|^{-u}\lambda(S,H,G)}.
\end{eqnarray*}
\end{lemma}
\begin{proof}
First of all, since $\mu_u$ is a measure and the sequence $U_{T_m, H^{T_m}}$ is descending, we have $$\mu_u(V_{S,H})=\mu_u( \cap_{m\geq 1} U_{T_m,H^{\bT_m}})=\lim_{m\to \infty} \mu_u (U_{T_m, H^{\bT_m}}) = \lim_{m\to \infty} \lim_{n \to \infty} \Prob((X_{u,n})^{\bT_m} \simeq H).$$
By definition, we have that $\lambda(T,H,G)$ is non-decreasing in $T$, i.e. if $T\sub T'$ then
$\lambda(T,H,G)\leq \lambda(T',H,G)$.  Further, again by definition, we have
$$
\lambda(S,H,G) =\lim_{m\ra\infty} \lambda(T_m,H,G).
$$

When $m$ is sufficiently large such that $H$ is level $T_m$, we have
\begin{eqnarray*}
&&\lim_{n\ra\infty} \Prob((X_{u,n})^{\bT_m}\isom H)\\
&=& \frac{1}{|\Aut(H)||H|^{u}}\prod_{\substack{G \in \mathcal{A}_H}} 
\prod_{i=1}^{\infty} (1-\lambda(T_m,H,G)
\frac{h_H(G)^{-i}}{ |G|^{{u}}})
 \prod_{\substack{ G\in \mathcal{N}}}  e^{-|G|^{-u}\lambda(T_m,H,G)}
\end{eqnarray*}
by Equation~\eqref{E:lim} since $T_m$ is finite. For each $G\in \mathcal{A}_H$, the factor
$$
\prod_{i=1}^{\infty} (1-\lambda(T_m,H,G)
\frac{h_H(G)^{-i}}{ |G|^{{u}}})
$$
is a limit of terms
$$
\prod_{i=1}^{m(T_m,n,H,G)} (1-\frac{h_H(G)^{m(T_m,n,H,G)} h_H(G)^{-i}}{ |G|^{{n+u}}})
$$
by Lemma~\ref{L:seq},
each of which is a probability (see Corollary~\ref{C:abprob}) and hence in the interval $[0,1]$.
Since the factors
$$
\prod_{i=1}^{\infty} (1-\lambda(T_m,H,G)
\frac{h_H(G)^{-i}}{ |G|^{{u}}}) 
\quad \quad \textrm{and} \quad \quad e^{-|G|^{-u}\lambda(T_m,H,G)}
$$
are all in $[0,1]$ and are non-increasing in $m$, we have the second equality in the following
\begin{eqnarray*}
& &\lim_{m\ra\infty}\lim_{n\ra\infty} \Prob((X_{u,n})^{\bT_m}\isom H)\\
&=& \lim_{m\ra\infty} \frac{1}{|\Aut(H)||H|^{u}}\prod_{\substack{G \in \mathcal{A}_H}} 
\prod_{i=1}^{\infty} (1-\lambda(T_m,H,G)
\frac{h_H(G)^{-i}}{ |G|^{{u}}})
  \prod_{\substack{ G\in \mathcal{N}}}  e^{-|G|^{-u}\lambda(T_m,H,G)}\\
 &=& \frac{1}{|\Aut(H)||H|^{u}}\prod_{\substack{G \in \mathcal{A}_H}} \lim_{m\ra\infty}
\prod_{i=1}^{\infty} (1-\lambda(T_m,H,G)
\frac{h_H(G)^{-i}}{ |G|^{{u}}})
  \prod_{\substack{ G\in \mathcal{N}}} \lim_{m\ra\infty} e^{-|G|^{-u}\lambda(T_m,H,G)}\\
  &=& \frac{1}{|\Aut(H)||H|^{u}}\prod_{\substack{G \in \mathcal{A}_H}} 
\prod_{i=1}^{\infty} (1-\lambda(S,H,G)
\frac{h_H(G)^{-i}}{ |G|^{{u}}})
  \prod_{\substack{ G\in \mathcal{N}}}  e^{-|G|^{-u}\lambda(S,H,G)}.
\end{eqnarray*}
The last equality uses the continuity from Lemma~\ref{L:seq}.
\end{proof}

\begin{theorem}\label{T:calcinf}
The statement in Theorem~\ref{T:calc} also works for an arbitrary set $S$ of finite groups.
\end{theorem}

\begin{proof}
 Let $T_m$ be the subset of $S$ of all groups of order at most $m$ in $S$.
 Since $H$ is level $S$, for large enough $m$ we have that $ H^{\bT_m}=H$, and from now on we only consider $m$ this large. 
 We can show that $G^{\bS}\isom H$ if and only if for every $m\geq 1$ we have $G^{\bT_m}\isom H^{\bT_m}.$
Since  $G^{\bar{T}_m}$ is a quotient of $G^{\bar{S}}$, the ``only if'' direction is clear.  If we take the inverse limit of the sets $\operatorname{Isom}(G^{\bar{T}_m},H^{\bar{T}_m})$, with the natural maps, we have an inverse limit of non-empty finite sets, which is non-empty.  
An element of this inverse limit gives us an isomorphism $G^{\bar{S}}\isom H^{\bar{S}}$.
 
 From this, and the basic properties of a measure, and Equation~\eqref{E:fixedn} for finite $S$, we have that 
\begin{eqnarray*}
 &&\Prob((X_{u,n})^{\bS}\isom H)\\
 &=& \lim_{m\ra\infty} \frac{|\Sur(\hat{F}_n,H)|}{|\Aut(H)||H|^{n+u}}\prod_{\substack{G \in \mathcal{A}_H}} 
\prod_{k=0}^{m(T_m,n,H,G)-1} (1-\frac{h_H(G)^k}{ |G|^{{n+u}}})
 \prod_{\substack{ G\in \mathcal{N}}}  (1-|G|^{-{n-u}})^{m(T_m,n,H,G)}.
\end{eqnarray*}
 From  Lemma~\ref{L:mincinS}, we have that
 $$
 \prod_{k=0}^{m(T_m,n,H,G)-1} (1-\frac{h_H(G)^k}{ |G|^{{n+u}}}) \quad \quad \textrm{and} \quad \quad 
 (1-|G|^{-{n-u}})^{m(T_m,n,H,G)}
 $$
are non-increasing in $m$, and as they are probabilities they are in the interval $[0,1]$.  Thus it follow from basic analysis that
\begin{align*}
& \Prob((X_{u,n})^{\bS}\isom H)\\= & \frac{|\Sur(\hat{F}_n,H)|}{|\Aut(H)||H|^{n+u}}\prod_{\substack{G \in \mathcal{A}_H}} 
\lim_{m\ra\infty} \prod_{k=0}^{m(T_m,n,H,G)-1} (1-\frac{h_H(G)^k}{ |G|^{{n+u}}})
 \prod_{\substack{ G\in \mathcal{N}}} \lim_{m\ra\infty} (1-|G|^{-{n-u}})^{m(T_m,n,H,G)}
\end{align*}
By definition of $m(S,n,H,G)$, we have that $\lim_{m\ra\infty} m(T_m,n,H,G)=m(S,n,H,G)$ (and the latter is finite).
  Thus, we have
  \begin{align*}
& \Prob((X_{u,n})^{\bS}\isom H)\\= & \frac{|\Sur(\hat{F}_n,H)|}{|\Aut(H)||H|^{n+u}}\prod_{\substack{G \in \mathcal{A}_H}} 
\prod_{k=0}^{m(S,n,H,G)-1} (1-\frac{h_H(G)^k}{ |G|^{{n+u}}})
 \prod_{\substack{ G\in \mathcal{N}}}  (1-|G|^{-{n-u}})^{m(S,n,H,G)},
\end{align*}
which is  Equation~\eqref{E:fixedn} for arbitrary $S$.

Next, towards Equation~\eqref{E:lim} for arbitrary $S$, we will show that the order of the limits in Lemma~\ref{L:wrongorder} could be exchanged.

For every $m$, we have $\Prob((X_{u,n})^{\bS}\isom H)\leq \Prob((X_{u,n})^{\bT_m}\isom H)$ and so
\begin{eqnarray}
&&\limsup_{n\ra\infty} \Prob((X_{u,n})^{\bS}\isom H)\notag\\
&\leq& \lim_{m\ra\infty} \lim_{n\ra\infty} \Prob((X_{u,n})^{\bT_m}\isom H)\notag\\
&=& \frac{1}{|\Aut(H)||H|^{u}}\prod_{\substack{G \in \mathcal{A}_H}} \prod_{i=1}^{\infty} (1-\lambda(S,H,G)
\frac{h_H(G)^{-i}}{ |G|^{{u}}})
  \prod_{\substack{ G\in \mathcal{N}}}  e^{-|G|^{-u}\lambda(S,H,G)}.\label{E:doublim}
\end{eqnarray}

From here we consider two cases.
Case 1 will be the following:
$$
\sum_{G\in \mathcal{A}_H } \frac{\lambda(S,H,G)}{h_H(G)|G|^u} +\sum_{G\in  \mathcal{N} } \frac{\lambda(S,H,G)}{|G|^u} \quad \textrm{diverges}.
$$
In case $1$, the product in Equation~\eqref{E:doublim} is $0$,  and we have proven $\lim_{n\ra\infty} \Prob((X_{u,n})^{\bS}\isom H)=0$, establishing Equation~\eqref{E:lim}.

Case 2 will be the following:
$$
\sum_{G\in \mathcal{A}_H } \frac{\lambda(S,H,G)}{h_H(G)|G|^u} +\sum_{G\in  \mathcal{N} } \frac{\lambda(S,H,G)}{|G|^u} \quad \textrm{converges}.
$$
We define a \emph{minimal} non-trivial $H$-extension $(E,\pi)$ to be an $H$-extension whose only quotient $H$-extensions are itself and the trivial one.  These are exactly the $H$-extensions with $\ker \pi$ an irreducible $E$-group (under conjugation).  %, by minimality.  
 Also, these are exactly the $H$-extensions $(E,\pi)$ such that $\ker \pi$ is an abelian irreducible $H$-group or $\ker \pi$ is a power of a non-abelian simple group and an irreducible $E$-group.
  Since $|\Aut(E)|\geq |\Aut_H(E,\pi)|$ and $h_H(G)\geq 2$,  we have
$$
\sum_{G\in \mathcal{A}_H } \frac{\lambda(S,H,G)}{h_H(G)|G|^u} +\sum_{G\in  \mathcal{N} } \frac{\lambda(S,H,G)}{|G|^u} 
 \geq  \frac{1}{2}\sum_{\substack{(E,\pi) \textrm{ min. non-triv. $H$-extension  }
\\ \textrm{$E$ level $S$}}}|\Aut(E)|^{-1} |G|^{-u}. 
$$
Since we are in case 2, the sum on the right converges, and
\begin{eqnarray*}
&&\lim_{m\ra\infty} \sum_{\substack{(E,\pi) \textrm{ min. non-triv. $H$-extension  }\\
\textrm{$E$ level $S$, but not level $T_m$}}}|\Aut(E)|^{-1} |E|^{-u}\\
&=&|H|^{-u}\lim_{m\ra\infty} \sum_{\substack{(E,\pi) \textrm{ min. non-triv. $H$-extension  }\\
\textrm{$E$ level $S$, but not level $T_m$}}}|\Aut(E)|^{-1} |G|^{-u}  \\
&=&0.
\end{eqnarray*}

   If $(X_{u,n})^{\bS}\not\isom H,$ but $(X_{u,n})^{\bT_m}\isom H$ for some $m$, then $X_{u,n}$ has a surjection to $H$ and thus $X_{u,n}$ has a surjection to some minimal non-trivial $H$-extension $(E,\pi)$ of level $S$ but not level $T_m$.
Note that
\begin{align*}
\Prob(X_{u,n} \textrm{ has a surjection to $E$})&\leq \E(\textrm{quotients of $X_{u,n}$ isom. to $E$})\\
&= |\Aut(E)|^{-1} \E( |\Sur(X_{u,n},E)|)\\
&= |\Aut(E)|^{-1} \frac{|\Sur(\hF_n,E)|}{|E|^{n+u}}\\
&\leq |\Aut(E)|^{-1} |E|^{-u}.
\end{align*}

Thus,
\begin{align*} 
\Prob((X_{u,n})^{\bS}\isom H) \geq \Prob((X_{u,n})^{\bT_m}\isom H) -\sum_{\substack{(E,\pi) \textrm{ min. non-triv. $H$-extension  }\\
\textrm{$E$ level $S$, but not level $T_m$}}}|\Aut(E)|^{-1} |E|^{-u} 
\end{align*}
and 
\begin{eqnarray*} 
&&\liminf_{n\ra\infty} \Prob((X_{u,n})^{\bS}\isom H) \\
&\geq& \lim_{n\ra\infty} \Prob((X_{u,n})^{\bT_m}\isom H) -\sum_{\substack{(E,\pi) \textrm{ min. non-triv. $H$-extension  }\\
\textrm{$E$ level $S$, but not level $T_m$}}}|\Aut(E)|^{-1} |E|^{-u} 
\end{eqnarray*} 
Now we take a $\lim_{m\ra\infty}$ of both sides and  conclude Equation~\eqref{E:lim} for arbitrary $S$.
Finally, note that if $m(S,n,H,G)\ne 0$, then $m(T_m,n,H,G)\ne 0$ for some $m$, and so the last statement of Theorem~\ref{T:calc} for infinite $S$ follows from the same statement for finite $S$. 
\end{proof}
 
Though this doesn't follow from weak convergence (see Proposition~\ref{P:alltrivial} and Remark~\ref{R:alltrivial}, for example), we see here that $\mu_u$ and $\lim_{n\ra\infty} \mu_{u,n}$ agree on the $V_{S,H}$.
\begin{corollary}\label{C:coninfS}
  Let $S$ be a set of finite groups and $H$ a finite level $S$ group. Then we have 
  $$\lim_{n\to \infty}\mu_{u,n} (V_{S,H}) = \mu_u(V_{S,H}).$$
\end{corollary}

\begin{proof}
In the proof of Theorem~\ref{T:calcinf}, we showed that 
$$\lim_{n\to \infty} \lim_{m\to \infty} \Prob((X_{u,n})^{\bT_m} \simeq H) = \lim_{m \to \infty} \lim_{n \to \infty} \Prob((X_{u,n})^{\bT_m}\simeq H).$$
By Lemma~\ref{L:wrongorder}, the right-hand side in the above equation is $\mu_u(V_{S,H})$. Also, since $\mu_{u,n}$ are measures on $\mathcal{P}$, we have
$\lim_{m\to \infty}\Prob((X_{u,n})^{\bT_m}\simeq H)= \mu_{u,n}(V_{S,H})$.
\end{proof}

\section{Examples of the values of $\mu_u$}\label{S:examples}

In this section, we will apply Theorem~\ref{T:calcinf} to compute $\mu_u(A)$ for some interesting Borel sets $A$.

\begin{example}[Trivial group]
Let $S$ contain every finite group. Then the trivial group is the only element in $V_{S,1}$. By Lemma~\ref{L:HfactR}, if $(E,\pi)$ is an extension of the trivial group such that $\ker \pi$ is irreducible $E$-group, then $E$ is a finite simple group. Then it follows from the definition of $\lambda(S,H,G)$ that
$$\lambda(S,1,G)=\begin{cases}
  1 & G \text{ is a non-trivial abelian simple group}\\
  |\Aut(G)|^{-1} & G \text{ is a non-abelian simple group}\\
  0 & otherwise,
\end{cases}
$$
where in the first case, we use the fact that $h_1(G)-1=|\Hom(G,G)|-1=|\Aut(G)|$ as $G$ is simple. 
By Theorem~\ref{T:calcinf}, we have 
$$\mu_u(\text{trivial group})=\prod_{p\text{ prime}}\prod_{i=u+1}^{\infty}(1-p^{-i}) \prod_{\substack{G \text{ finite simple}\\ \text{non-abelian group}}} e^{-|G|^{-u}|\Aut(G)|^{-1}}.$$
The above product over prime integers is zero if and only if $u\leq 0$. When $u\geq 1$, by the classification of finite simple groups, the number of finite simple groups of given order is at most 2. Note that $|\Aut(G)|\geq |\Inn(G)|=|G|$ for every non-abelian simple group $G$. We have
$$\prod_{\substack{G \text{ finite simple}\\ \text{non-abelian group}}} e^{-|G|^{-u}|\Aut(G)|^{-1}} \geq \exp (-\sum_{\substack{G \text{ finite simple}\\ \text{non-abelian group}}}|G|^{-u-1}) > 0,$$
which shows that $\mu_u(\text{trivial group})>0$ if and only if $u\geq 1$. By using the classification of finite simple groups, we are able to give the following approximations
$$\mu_u(\text{trivial group})\approx\begin{cases}
 0.4357 & \text{when }u=1\\
 0.7168 & \text{when }u=2\\
 0.8616 & \text{when }u=3.
\end{cases}$$
We observe that the product over non-abelian factors is very close to 1 and cannot be seen in this many digits.
\end{example}

\begin{example}[Any infinite group]
Again let $S$ contain all finite groups. Let $H$ be an infinite profinite group in $\mathcal{P}$, and $H_{\ell}$ denote the pro-$\bS_{\ell}$ completion of $H$. Since $U_{S_\ell, H_{\ell}}$ is a sequence of basic opens that is decreasing in $\ell$ and $\cap_{\ell} U_{S_{\ell}, H_{\ell}} = \{H\}$, we obtain 
\begin{eqnarray*}
 \mu_u(\{H\})&=& \lim_{\ell \to \infty} \mu_u(U_{S_\ell, H_\ell})\\
 &\leq& \lim_{\ell\to \infty} \frac{1}{|\Aut(H_\ell)| |H_\ell|^u}.
\end{eqnarray*}
Note that $H$ is the inverse limit of $H_{\ell}$, so $\lim_{\ell\to \infty} |H_\ell| =\infty$. It follows that $\mu_u(\{H\})=0$ when $u\geq 1$. When $u=0$, since $\{H\}$ is contained in the Borel set $A:=V_{\{\text{all abelian groups}\}, H^{ab}}$ and $\mu_0(A)=0$ (see Example~\ref{E:abelian}), we have $\mu_0(\{H\})=0$.
\end{example}

\begin{example}[Pro-$p$ abelianization]\label{ex:abelian}
Let $p$ be a prime integer and $S$ the set consisting of all finite abelian $p$-groups. Then $\bS=S$ and $(\Z/p\Z, 1)$ is the only element in $\CF(\bS)$. Let $H$ be a finite abelian $p$-group of generator rank $d$. Then for any $G\in\mathcal{A}_H\cup \mathcal{N}$, the factor in Theorem~\ref{T:calc} associated to $G$ is $1$, unless $G=\Z/p\Z$ with the trivial $H$-action.
We consider the Borel set $V_{S,H}$ that is the set of all profinite groups whose maximal abelian pro-$p$ quotient is $H$. For any integer $n\geq d$ , there is a normal subgroup $N$ of $(\hat{F}_n)^{\bS}=(\Z_p)^n$ such that the corresponding quotient is $H$. Since $H$ is finite, $N$ is isomorphic to $(\Z_p)^n$ with the trivial $(\hat{F}_n)^{\bS}$-action, which shows that $m(S,n,H,\Z/p\Z)=n$ and $\lambda(S,H,\Z/p\Z)=1$. By Lemma~\ref{L:wrongorder}, we have
$$\mu_u(V_{S,H})=\frac{1}{|\Aut(H)||H|^u} \prod_{i=1}^{\infty}(1-p^{-i-u}).$$
When $u<0$, this probability is 0, which is as expected since we can never get a finite quotient of $(\Z_p)^n$ with fewer than $n$ relators. When $u\geq 0$, we get a finite group with probability 1. When $u=0$ or $1$, these are the measures used in the Cohen-Lenstra heuristics for class groups of quadratic number fields.

More generally, let's consider an infinite abelian pro-$p$ group $H$ in $\mathcal{P}$. Since $H\in \mathcal{P}$, the pro-$\overline{\{\Z/p\Z\}}$ completion of $H$ is finite, so $H$ is finitely generated, i.e. $H=H_1 \times (\Z_p)^r$ for a finite abelian $p$-group $H_1$ and a positive integer $r$. Let $T_j:=\{\Z/p^j\Z\}$. So $\bT_j$ is an increasing sequence and $\cup \bT_j=S$. Assume $n\geq d$ and $j$ is greater than the exponent of $H_1$. Then we have
 $$(\hat{F}_n)^{\bT_j}=(\Z/p^j\Z)^n \text{  and  } H^{\bT_j}=H_1 \times (\Z/p^j\Z)^r.$$
So $m(T_j, n,H, \Z/p\Z)=n-r$, $\lambda(T_j, H, \Z/p\Z)=p^{-r}$ and 
\begin{eqnarray*}
\mu_u(V_{T_j,H}) &=& \frac{1}{|\Aut(H^{\bT_j})||H_1|^u p^{jru}} \prod_{i=1+u+r}^{\infty}(1-p^{-i})\\
&=&\frac{1}{|\Aut(H_1)||H_1|^{2r+u}p^{jr(r+u)}} \prod_{i=1}^r(1-p^{-i})^{-1}\prod_{i=1+u+r}^{\infty}(1-p^{-i}),
\end{eqnarray*}
since $$|\Aut(H^{\bT_j})|
= |\Aut(H_1)||H_1|^{2r} p^{jr^2} \prod_{i=1}^r (1- p^{-i}).$$
It follows that $\mu_u(V_{S,H})=\lim_{j\to \infty}(V_{T_j,H})> 0$ if and only if $u+r=0$, in which case
$$\mu_u(V_{S,H})=\frac{1}{|\Aut(H_1)||H_1|^{-u}} \prod_{i=1-u}^{\infty} (1-p^{-i}).$$
So we see that when $u<0$, $\mu_u(V_{S, H})>0$ if and only if the (torsion-free) rank of $H$ is $-u$ and we get the groups in such form with probability 1.
\end{example}

\begin{example}[Abelianization]\label{E:abelian}
Similar to the example above, when $S$ is the set of all finite abelian groups and $H$ is a finite abelian group, we have
$$
\mu_u(V_{S,H})=\frac{1}{|\Aut(H)||H|^u} \prod_{p \textrm{ prime}}\prod_{i=1}^{\infty}(1-p^{-i-u}),
$$
which is $0$ if $u\leq 0$ and is positive if $u\geq 1$.  If $H=H_1 \times (\hat{\Z})^r$, then 
$$\mu_u(V_{S,H})=\frac{1}{|\Aut(H_1)||H_1|^{-u}} \prod_{p \textrm{ prime}} \prod_{i=1-u}^{\infty} (1-p^{-i})>0$$
if $u=-r<0$ and $\mu_u(V_{S,H})=0$ otherwise.
\end{example}

In order to consider the pro-$p$ quotients of our random groups, we will first need to recall the definitions of some $p$-group invariants.
Let $H$ be a finite $p$-group of generator rank $d$. The relation rank $r(H)$ of $H$ is defined to be the smallest number of relations in a pro-$p$ presentation of $H$ (and also it is known that $r(H)=\dim_{\F_p}\operatorname{H}^2(H, \Z/p\Z)$). 
Let $1\to N\to \hat{F}_d \to H \to 1$ be a presentation of $H$. Define $N^*:=[N,\hat{F}_d]\cdot N^p$. Then $N^*$ is the minimal $\hat{F}_d$-normal subgroup of $N$ such that $N/N^*$ is a finite elementary abelian $p$-group with trivial $\hat{F}_d/N^*$-action. $\hat{F}_d/N^*$ is called the \emph{$p$-covering group} of $H$, and $N/N^*$ is called the \emph{$p$-multiplicator} of $H$, and $\dim _{\F_p}(N/N^*)$ is called the \emph{$p$-multiplicator rank} of $H$.  It is not hard to see that  $r(H)$ is the $p$-multiplicator rank of $H$.

\begin{lemma}\label{L:p-multi}
  Let $H$ be a finite $p$-group of generator rank $d$, $S$ the set of all finite $p$-groups, and $G$ the $H$-group that is isomorphic to $\Z/p\Z$ with trivial $H$-action. Then $m(S,d,H,G)=r(H)$ and $m(S,n,H,G)=r(H)+n-d$ for every $n\geq d$.
\end{lemma}

\begin{proof}
  Since the intersection of every normal subgroup and the center of a finite $p$-group is nontrivial, every finite $p$-group acts trivially on all of its minimal normal subgroups, which implies $\CF(\bS)=\{(\Z/p\Z, 1)\}$. Recall that $m(S,n,H,G)$ is defined to be $\lim_{i\to\infty} m(T_i, n, H, G)$, where $T_i$ is an increasing sequence of finite sets of groups such that $\cup T_i=S$. When $i$ is sufficiently large such that $T_i$ contains the $p$-covering group of $H$, the map $\rho:(\hat{F}_d)^{\bT_i}\to H$ factors through the $p$-covering group of $H$. Let $1\to R \to F \to H \to 1$ be the fundamental short exact sequence associated to $T_i, d, H$. It is not hard to check that $R$ is also the maximal quotient of $\ker \rho$ that is an elementary abelian $p$-group with the trivial $F$-action. Therefore, $R$ is the $p$-multiplicator  of $H$ and $m(S,d,H,G)=m(T_i, d,H,G)=r(H)$.
  
  Assume $n\geq d$. We can find a surjection $\rho_1:\hat{F}_{n+1}\to H$ and generators $x_1, \cdots, x_{n+1}$ of $\hat{F}_{n+1}$ such that $\rho_1(x_{n+1})=1$. Let $\rho_2$ be the restriction of $\rho_1$ on the subgroup generated by $x_1, \cdots, x_n$. Then $\rho_2: \hat{F}_n\to H$ is a surjection. Let $1\to R_1 \to F_1 \overset{\pi_1}{\to} H \to 1$ and $1\to R_2 \to F_2 \overset{\pi_2}{\to} H \to 1$ be the fundamental short exact sequences associated to $T_i, n+1, H$ and $T_i, n, H$ that arise from $\rho_1$ and $\rho_2$ respectively. These constructions allow us to get a surjection $\pi: F_1 \to F_2$ with $\pi_1=\pi_2 \circ \pi$, and a generator set $y_1, \cdots, y_{n+1}$ of $F_1$ such that $\pi(y_{n+1})=1$. Since $y_{n+1}\in\ker \pi_1$ and $F_1$ acts trivially on $R_1$, the subgroup generated by $y_{n+1}$, which is isomorphic to $\Z/p\Z$, is normal in $F_1$. It implies that $R_1\simeq R_2 \times \Z/p\Z$ and $m(T_i, n+1, H, G)=m(T_i, n, H, G)+1$. By induction on $n$, we finish the proof of the lemma.
\end{proof}

\begin{example}[Pro-$p$ quotient]
  Let $H$ be a finite $p$-group of generator rank $d$, and $S$ the set of all finite $p$-groups, and $G$ the $H$-group that is isomorphic to $\Z/p\Z$ with trivial $H$-action. Since $\CF(\bS)=(\Z/p\Z,1)$, we have
$$\mu_{u}(V_{S,H})=\frac{1}{|\Aut(H)||H|^u} \prod_{i=1}^{\infty}(1-\frac{\lambda(S,H,G)}{p^{i+u}}).$$
By Equation \eqref{E:relmandlam} and Lemma \ref{L:p-multi}, $\lambda(S,H,G)=p^{r(H)-d}$. So 
$$\mu_u(V_{S,H})=\frac{1}{|\Aut(H)||H|^u}\prod_{i=1+u-r(H)+d}^{\infty} (1-p^{-i}),$$
and $\mu_u(V_{S,H})>0$ if and only if $u\geq r(H)-d$.  

Given $u$, if $X_{n,u}^{\bS}$ has generator rank $d$ with 
 $d^2/4\geq d+u$, we have that the $X_{n,u}^{\bS}$ is necessarily infinite by the Golod-Shafarevich inequality.  We can see from the pro-$p$ abelianization that we get groups $X_{n,u}^{\bS}$ with each generator rank $d\geq \min(0,-u)$ with positive probability. 
 All groups in $\mathcal{P}$ have their pro-$p$ quotient finitely generated. %Finitely presented?
\end{example}

\begin{example}[Pro-nilpotent quotient]
When $S$ is the set of all finite nilpotent groups and $H$ is a finite nilpotent group with Sylow $p$-subgroup $H_p$ of generator rank $d_p$, we have
$$\mu_u(V_{S,H})=\frac{1}{|\Aut(H)||H|^u}\prod_{p \textrm{ prime}}\prod_{i=1+u-r(H_p)+d_p}^{\infty} (1-p^{-i}).$$
Let $W$ be the set of profinite groups $G$ such that there are only finitely many primes $p$ such that the maximal pro-$p$ quotient of $G$ has generator rank $\geq max(2,-u+1)$.
By the Borel-Cantelli lemma, we can see that $\mu_u(W)=1$, and thus $\mu_u$ assigns probability $1$ to the set of groups who pro-nilpotent quotient is finitely generated.
\end{example}

\begin{example}[All infinite groups]\label{Ex:inf}
When $u\leq 0$, we have $\mu_u( \{ \textrm{infinite groups} \})=1$ (which can already be seen on the abelianization).  
When $u>0$, we have $0<\mu_u( \{ \textrm{infinite groups} \})<1$, since $\mu_u( \{ \textrm{trivial group} \})>0$ and there is positive probability of infinite pro-p quotient.
This was seen for the $\mu_{u,n}$ in \cite{Jarden2006}. 

\end{example}

\section{Which groups appear?}\label{S:prob0}
In this section, we consider the question of when $\mu_u$ is $0$ on our basic opens $U_{S,H}$.  
In order for a basic open $U_{S,H}$ to have positive probability for $\mu_{u,n}$, the group $H$ needs to be able to be generated as a pro-$\bS$ group with $n$ generators and $n+u$ relations.  We will see in Proposition~\ref{P:prob0} that the same criterion holds for 
 $\mu_u$.
We start with a lemma about the number of generators and relations required to present a pro-$\bS$ group.

\begin{lemma}\label{L:present}
Let $S$ be a finite set of finite groups and $u$ an integer.  Let $H$ be a finite pro-$\bS$ group that can be generated by $d$ generators.  If $H$ can be presented as a pro-$\bS$ group by $m$ generators and $m+u$ relations, then $H$ can be presented as a pro-$\bS$ group by $d$ generators and $d+u$ relations.
\end{lemma}

This is the same as the situation when $S$ is the set of all profinite groups and $H$ is a finite group (see \cite[Theorem 0.1]{Lubotzky2001}), but contrasts to the more general situation of presenting $H$ as a finite group, where the analog is a long-standing open question (see \cite[Lecture 1: Question 3]{Gruenberg1976}).

\begin{proof}
Suppose for the sake of contradiction that we have a counterexample, and consider one with $m$ minimal.  We have that
\begin{eqnarray*}
&& \mu_{u,m}(U_{S,H})\\
&=& \frac{|\Sur(\hat{F}_m,H)|}{|\Aut(H)||H|^{m+u}}\prod_{\substack{G \in \mathcal{A}_H}} 
\prod_{k=0}^{m(S,m,H,G)-1} (1-\frac{h_H(G)^k}{ |G|^{{m+u}}})
 \prod_{\substack{ G\in \mathcal{N}}}  (1-|G|^{-{m-u}})^{m(S,m,H,G)} \\
 &>& 0.
\end{eqnarray*}
In particular, since $|G|$ is a power of $h_H(G)$ this implies that for $G\in \mathcal{A}_H$, we have
$$
h_H(G)^{m(S,m,H,G)-1}|G|^{-m-u}\leq h_H(G)^{-1}.
$$
However, since we have a minimal counterexample, we have that $m>d$ and
\begin{eqnarray*}
&&\mu_{u,m-1}(U_{S,H})\\
&=& \frac{|\Sur(\hat{F}_{m-1},H)|}{|\Aut(H)||H|^{m-1+u}}\prod_{\substack{G \in \mathcal{A}_H}} 
\prod_{k=0}^{m(S,m-1,H,G)-1} (1-\frac{h_H(G)^k}{ |G|^{{m-1+u}}})
 \prod_{\substack{ G\in \mathcal{N}}}  (1-|G|^{-m+1-u})^{m(S,m-1,H,G)}\\
 &=&0.
\end{eqnarray*}
By the final statement of Theorem~\ref{T:calc}, we have that one of the factors is $0$.
Since $m>d$, we have $|\Sur(\hat{F}_{m-1},H)|\ne 0$.  If $H$ is the trivial group, the lemma is clear.  Thus we can assume $d\geq 1$ and $m\geq 2$, %.  If $m-1+u=0$, then $m=1$ and $m(S,m-1,H,G)=0$, 
and so for $G\in\mathcal{N}$ we have $(1-|G|^{-m+1-u})^{m(S,m-1,H,G)}>0$.
Thus, for some $G\in \mathcal{A}_H$, we have 
\begin{equation}\label{eq:001}
h_H(G)^{m(S,m-1,H,G)-1}|G|^{-m+1-u}\geq 1.
\end{equation}

If $\rho_n : (\hat{F}_n)^{\bar{S}} \ra H$ is a surjection, we have
\begin{align*}
(h_H(G)^{m(S,n,H,G)}-1)|G|^{-n}
&= (h_H(G)-1) \sum_{\substack{\textrm{isom. classes of $H$-extensions $(E,\pi)$}\\ \textrm{$\ker \pi \simeq G$}
\\ \textrm{$\ker \pi$ irred. $E$-group} \\
\textrm{$E$ is level $S$} }} \frac{|\Sur(\rho_n, \pi)|}{|\Aut_H(E,\pi)||G|^n} .
\end{align*}
Any surjection from $\rho_n$ to $\pi$ can be extended to a surjection from $\rho_{n+1}$ to $\pi$ in $|G|$ ways.
So, $(h_H(G)^{m(S,n,H,G)}-1)|G|^{-n}$ is non-decreasing in $n$.
So we have
\begin{equation}\label{eq:002}
\frac{h_H(G)^{m(S,m,H,G)}-1}{|G|^{m}} \geq \frac{h_H(G)^{m(S,m-1,H,G)}-1}{|G|^{m-1}} ,
\end{equation}
and then we have
\begin{align*}
|G|^uh_H(G)&\leq h_H(G)^{m(S,m-1,H,G)}|G|^{-m+1} \\
&\leq h_H(G)^{m(S,m,H,G)}|G|^{-m} +|G|^{-m+1}-|G|^{-m} \\
&\leq |G|^u +|G|^{-m+1}-|G|^{-m},
\end{align*}
where the first and the last inequalities follow by \eqref{eq:001} and the second one follows by \eqref{eq:002}.
Since $m>d$ and $H$ can be generated by $d$ generators, the number of relations $m+u$ has to be positive. From above, we have $|G|^{m+u}(h_H(G)-1)\leq(|G|-1)$. Then this is a contradiction, since  $h_H(G)\geq 2$.
\end{proof}

Lemma~\ref{L:present} leads to the following definition.

\begin{definition}
Let $S$ be a finite set of finite groups and $u$ be an integer.  We call a finite group $H$ with generator rank $d$ \emph{achievable} (with $S$ and $u$ implicit) if it can be generated as a pro-$\bar{S}$ group with $d$ generators and $d+u$ relations.
\end{definition}

\begin{proposition}\label{P:prob0}
Let $S$ be a finite set of finite groups and $u$ be an integer.  Then for a finite group $H$ we have
that
$\mu_u(U_{S,H})>0$ if $H$ is achievable and $\mu_u(U_{S,H})=0$ otherwise.
\end{proposition}

So given $u$, our measure $\mu_u$ is supported on those groups in $\mathcal{P}$ whose pro-$\bar{S}$ completion is achievable (for $u,S$) for every finite set $S$ of finite groups.
Note that given $S$, any finite pro-$\bar{S}$ group $H$ is achievable for $u$ sufficiently large. 
 
\begin{example}
From \cite[Theorem A]{Guralnick2007}, we have that every finite simple group can be presented as a profinite group with $2$ generators and $18$ relations.
Thus if $u\geq 16$ and $S$ is a finite set of finite groups with $H\in \bS$ a simple group, then $H$ is achievable.
\end{example}

\begin{example}
If $S$ is the set of all groups of order $32$ and $u\leq 0$, we can see that $H=\Z/2\Z \times \Z/2\Z$ is not achievable.  To obtain $H$ as a quotient of $F_2$, it is easy to compute we need at least 3 relations (for both generators to be order $2$ and for them to commute with each other).
\end{example}

\begin{remark}
Proposition~\ref{P:prob0} need not hold for infinite $S$.  For example, if $S$ is the set of all finite abelian groups, then any finite abelian group $H$ can be presented as an abelian group with $n$ generators and $n$ relations, but $\mu_0(V_{S,H})=0$.  (See Example~\ref{E:abelian}.)  This is because the product over $G\in \mathcal{A}_H$ is contains factors $(1-p^{-1})$ for each prime $p$ and thus is $0$ even though no individual factor is $0$.  Further, if $S$ is the set of all groups and $H\in\mathcal{P}$, we have $\mu_0(V_{S,H})\leq \mu_0(V_{\textrm{\{abelian groups\}},H^{ab}}) =0$.  Some of those groups $H$ can be profinitely presented with $n$ generators and $n$ relations.
It is an interesting open question to understand in general for which infinite $S$ and finite $H$ does the product in Equation~\ref{E:lim} give $\mu_u(V_{S,H})=0$ even when none of the factors in the product is $0$.
\end{remark}

\begin{proof}[Proof of Proposition~\ref{P:prob0}]
By Lemma~\ref{L:present}, if $H$ is not achievable, then $\mu_{u,n}(U_{S,H})=0$ for all $n$ and hence $\mu_u(U_{S,H})=0$.
Suppose that $\mu_u(U_{S,H})=0$.  
Then using Theorem~\ref{T:calc} and Remark~\ref{R:finprod}, we must have that one of the factors in
\begin{eqnarray*}
%&&
%\lim_{n\ra\infty} \Prob((X_{u,n})^{\bS}\isom H)\\
%&=&
\frac{1}{|\Aut(H)||H|^{u}}\prod_{\substack{G \in \mathcal{A}_H}} 
\prod_{i=1}^{\infty} (1-\lambda(S,H,G)
\frac{h_H(G)^{-i}}{ |G|^{{u}}})
 \prod_{\substack{ G\in \mathcal{N}}}  e^{-|G|^{-u}\lambda(S,H,G)}.\nonumber
\end{eqnarray*}
is $0$, i.e. for some $G\in \mathcal{A}_H$ we have $\lambda(S,H,G)h_H(G)^{-1}|G|^{-u}\geq 1$.  Recall by Remark~\ref{R:Gpow} that $|G|$ is a power of $h_H(G)$ and thus so is $\lambda(S,H,G)$.  In fact, for sufficiently large $n$, we have $\lambda(S,H,G)=h_H(G)^{m(S,n,H,G)}|G|^{-n}$.
Thus, for sufficiently large $n$, we have $h_H(G)^{m(S,n,H,G)-1}\geq |G|^{n+u}$ and $\mu_{u,n}(U_{S,H})=0$.
However if we can present $H$ as a pro-$\bar{S}$ group with $d$ generators and $d+u-k$ relations with $k\geq 0$, we can add $m$ generators for any $m$ and $m+k$ relations to trivialize those generators, to present $H$ with $d+m$ generators and $d+m+u$ relations for all $m\geq 0$, which implies $\mu_{u,n}(U_{S,H})>0$ for $n$ sufficiently large.

\end{proof}

\section{Comparision to non-profinite groups}\label{S:non-profinite}
Let $Y_{u,n,\ell}$ be $F_n$ modulo $n+u$ random relations uniform from words of length at most $\ell$.    In this section, we will compare this model to our $X_{u,n}$.  To put the groups on the same footing,  we take the profinite completions $\hat{Y}_{u,n,\ell}$ of the $Y_{u,n,\ell}$.  Alternatively, we could enlarge our measure space to include non-profinite groups, with the same definition of basic opens.  Since our topology would not separate groups with the same profinite completion, we might as well consider only the profinite completions.  (Note by \cite{Ollivier2011} and \cite{Agol2013},  at density $<1/6$, these groups are asymptotically almost surely residually finite and thus inject into their profinite completions.)

The following is almost the same as \cite[Lemma 4.4]{Dunfield2006}, but we include it here for completeness.
\begin{proposition}\label{P:same}
Given integers $n,u$, we have that the distributions $\nu_{u,n,\ell}$ of the $\hat{Y}_{u,n,\ell}$ weakly converge to $X_{u,n}$.
\end{proposition}
\begin{proof}
As in the proof of Theorem~\ref{T:Main}, it suffices to show that for each finite group $H$ and finite set $S$ of groups that
$$
\lim_{\ell\ra\infty} \nu_{u,n,\ell}(U_{S,H})=\mu_{u,n}(U_{S,H}).
$$
Thus we are asked to compare the quotient of the finite group $(\hat{F}_n)^{\bar{S}}$ %(see Lemmas~\ref{L:sfinite} and \ref{L:redrel}) 
by the image of random uniform words of length at most $\ell$ versus by uniform random relators.  However as $\ell\ra\infty$ the image of a random uniform words of length at most $\ell$ converges to the uniform distribution on $(\hat{F}_n)^{\bar{S}}$, by the fundamental theorem on irreducible, aperiodic finite state Markov chains \cite[Theorem 6.6.4]{Durrett2010}.
\end{proof}

Next we see that taking a number of relations that is going to infinity always gives groups weakly converging to the trivial group in our topology.  This includes all positive density Gromov random groups as well as plenty of density $0$ random groups.

\begin{proposition}\label{P:alltrivial}
Let $u(\ell)$ be an integer valued function of the positive integers that goes to $\infty$ as $\ell\ra\infty$.
Then $\nu_{u(\ell),n,\ell}$ weakly converge to the probability measure supported on the trivial group as $\ell\ra\infty$.
\end{proposition}
\begin{proof}
Fix a finite set $S$ of finite groups and a finite group $H$.
Fix an integer $v$.  For $u(\ell)\geq v$, we have that 
$$
\Prob(\hat{Y}_{u(\ell),n,\ell} \textrm{ has a surjection to $H$})\leq \Prob(\hat{Y}_{v,n,\ell} \textrm{ has a surjection to $H$}).
$$
Since the set of groups with a surjection to $H$ is open and closed by Proposition~\ref{P:same}, we have that
$$
\lim_{\ell\ra\infty} \Prob(\hat{Y}_{v,n,\ell} \textrm{ has a surjection to $H$}) =\Prob(X_{v,n} \textrm{ has a surjection to $H$}).
$$
It is easy to see using the approach of our paper that
$$
\E(|\Sur(X_{v,n},H|)=\frac{|\Sur(F_n,H)|}{|H|^{n+v}}\leq |H|^{-v}.
$$
Thus
$$
\limsup_{\ell\ra \infty} \Prob(\hat{Y}_{u(\ell),n,\ell} \textrm{ has a surjection to $H$})
\leq |H|^{-v}
$$
for every $v$, and so $\lim_{\ell\ra \infty} \Prob(\hat{Y}_{u(\ell),n,\ell} \textrm{ has a surjection to $H$})=0$.
Thus, for every $U_{S,H}$ with $H$ non-trivial, we have that 
$$\lim_{\ell\ra \infty} \nu_{u(\ell),n,\ell}(U_{S,H})=0.$$

For $u(\ell)\geq v$, we have that 
$$
\Prob(\hat{Y}_{u(\ell),n,\ell}^{\bar{S}} \textrm{ trivial}) \geq \Prob(\hat{Y}_{v,n,\ell}^{\bar{S}} \textrm{ trivial}).
$$
By Proposition~\ref{P:same}, we have that
$$
\lim_{\ell\ra\infty}(\Prob(\hat{Y}_{v,n,\ell}^{\bar{S}} \textrm{ trivial}))=\Prob(X_{v,n}^{\bar{S}} \textrm{ trivial}).
$$
So
$$
\liminf_{\ell\ra\infty} \Prob(\hat{Y}_{u(\ell),n,\ell}^{\bar{S}} \textrm{ trivial}) \geq \limsup_{v\ra\infty} \Prob(X_{v,n}^{\bar{S}} \textrm{ trivial}).
$$
From Equation~\eqref{E:trivialgp}, we have that $\limsup_{v\ra\infty} \Prob(X_{v,n}^{\bar{S}} \textrm{ trivial})=1$.
(We can control the size of the product in Equation~\eqref{E:trivialgp}, for example, by using the fact that there are at most 2 finite simple groups of any particular order.) 
Thus, for every $U_{S,1}$, we have that 
$$\lim_{\ell\ra \infty} \nu_{u(\ell),n,\ell}(U_{S,1})=1.$$
\end{proof}

\begin{remark}\label{R:alltrivial}
Proposition~\ref{P:alltrivial} might seem surprising at first.   The groups $Y_{u(\ell),n,\ell}$ are plenty interesting as $\ell\ra\infty$.  In particularly they are asymptotically almost surely infinite at density $<1/2$ \cite{Gromov1993}, 
%Theorem 11 of \cite{Ollivier2005}
and residually finite at density $<1/6$ \cite{Ollivier2011,Agol2013}, and so have many  finite quotients.  The above shows that those quotients are escaping off to infinity, however.  %See \cite{Ollivier2005} for an introduction of the many properties of these groups that have been studied and their rich structure.  
Just as a very interesting sequence of numbers might go to $0$, an interesting sequence of random groups can converge to the trivial group.  A better analogy might be that a sequence of integers with interesting asymptotic growth that goes to $0$ $p$-adically.  This shows that, at low densities, the weak convergence of $\nu_{u(\ell),n,\ell}$  in $\ell$ is not as strong
as the convergence of the $\mu_{u,n}$ in $n$ that we see in Corollary~\ref{C:coninfS}.  In particular 
$$
\lim_{\ell\ra\infty}\nu_{u(\ell),n,\ell}(\textrm{trivial group})=0\ne 1.
$$

\end{remark}

%-----------------------------------------------------------------------
\section{List of notations appearing in multiple sections}\label{S:index}
%-----------------------------------------------------------------------

%\subsection*{Part~\ref{pt:general-conj}}

%Section~\ref{S:Notations}
%\vspace{1mm}

\begin{longtable}{|l|l|l|}
\hline
Notation & \S & Description \\
\hline
$\hF_n$ & \ref{S:intro} &  Free profinite group on $n$ generators \\
%\hline
$X_{u,n}$ & \ref{S:intro} & Random group with $n$ generators and $n+u$ Haar relators \\
$\mathcal{P}$ & \ref{S:intro} & $\left\{ \text{isom. cl. of profinite groups $G$} \,\Bigg|\, \begin{aligned}
	& |G^{\bS}|<\infty\, \forall \text{ finite set}\\
	& \text{$S$ of finite groups}
	\end{aligned} \right\}$\\
%\hline
$U_{S,H}$ & \ref{S:intro} & Basic open sets: $\{G\in \mathcal{P} \mid G^{\bS} \isom H\}$ for finite set $S$\\
$\mu_{u,n}$ & \ref{S:intro}, \ref{S:definemu} & distribution of $X_{u,n}$\\
%\hline
 $\mu_u$ & \ref{S:intro}, \ref{S:definemu} & probability measure given explicitly, limit of $\mu_{u,n}$\\
%\hline
%\hline
$\Cen_G(H)$ & \ref{SS:2.1} & Centralizer of a subgroup $H$ of $G$ \\
%\hline
$\Hom_F(G_1,G_2)$ & \ref{SS:2.2} & $F$-group homorphisms $G_1\to G_2$\\
%\hline
$\Sur_F(G_1, G_2)$ & \ref{SS:2.2} & $F$-group surjections $G_1\to G_2$\\
%\hline
$h_F(G)$ & \ref{SS:2.2} & $|\Hom_F(G,G)|$\\
%\hline
$[x_1, \cdots]_F$ & \ref{SS:2.2} & The closed normal $F$-subgroup of $G$ generated by $x_k$\\
%\hline
$(E, \pi)$ & \ref{SS:2.3} & $H$-extension $\pi: E\to H$\\
%\hline
$\Sur_H(\pi, \pi')$  & \ref{SS:2.3} & Surjections between $H$-extensions\\
%\hline
$\Aut_H(E, \pi)$ & \ref{SS:2.3} & Automorphism of an $H$-extensions $(E, \pi)$\\
%\hline
$\bar{S}$ & \ref{SS:proSdef} & Variety of groups generated by the set $S$\\
%\hline
$G^{\bar{S}}$ & \ref{SS:proSdef} & The pro-$\bar{S}$ completion of $G$\\
%\hline
$S_\ell$ & \ref{SS:proSdef} & $\{\text{isom. cl. of groups }G \mid |G|\leq \ell\}$\\
%\hline
%\hline
$\mathcal{A}$ & \ref{S:definemu} & The algebra of sets generated by basic opens $U_{S,H}$\\
%\hline
$\mathcal{A}_H$ & \ref{SS:basicq} & \{isom. cl. of non-trivial finite abelian irreducible $H$-groups\}\\
%\hline
$\mathcal{N}$ & \ref{SS:basicq} & \{isom. cl. of groups $G^j$ for nonabelian simple $G$ and $j\in \Z_{>0}$\}\\
%\hline
$\lambda(S,H,G)$ & \ref{SS:basicq} & Values defined for given $S$, $H$ and $G$\\
%\hline
$1\to R\to F\to H \to 1$ & \ref{S:setup} & Fundamental short exact sequence\\
%\hline
$\rho_M$ & \ref{S:factors} & $G\to \Aut(M)$ for a minimal normal subgroup $M$ of $G$\\
%\hline
$\CF(G)$, $\CF(T)$ & \ref{S:factors} & Chief factor pairs of a group $G$ or a set $T$\\
%\hline
$m(S,n,H,G)$ & \ref{S:basics} & Multiplicity of $G$ in $R$, see Section \ref{S:arbitraryS} for infinite $S$\\
%\hline
 $\mathcal{E}_H$ & \ref{S:basics} &poset of $H$-extensions\\
$\nu(D,E)$ & \ref{S:basics} & M\"obius function on a poset of $H$-extensions\\
%\hline
$V_{S,H}$ & \ref{S:arbitraryS} & $\{G\in \mathcal{P} \mid G^{\bS} \isom H\}$ for arbitrary set $S$\\
\hline
\end{longtable}
\bigskip

%Section~\ref{sec:blocks}
%\vspace{1mm}
%
%\begin{tabular}{|l|l|l|}
%\hline
%Notation & \S & Description \\
%\hline
%$\bk$ & \ref{ss:definitions-G} & Algebraically closed field of characteristic $p$ \\
%\hline
%\end{tabular}
%\bigskip

\subsection*{Acknowledgements} 
We thank Nigel Boston, Persi Diaconis, Tullia Dymarz, Benson Farb, Turbo Ho,  Peter Sarnak, Mark Shusterman, and Tianyi Zheng for helpful conversations.  We thank Goulnara Arzhantseva and the anonymous referee for feedback on an earlier version of this manuscript.
This work was done with the support of an American Institute of Mathematics Five-Year Fellowship, a Packard Fellowship for Science and Engineering, a Sloan Research Fellowship, a Vilas Early Career Investigator Award, and National Science Foundation grants DMS-1652116 and DMS-1301690.

\def\cprime{$'$}

\end{document}